\documentclass[11pt]{article}

\usepackage{macros}
\usebiggeometry
\bluehyperref
\usepackage{xr}
\usepackage[numbers]{natbib}

\newcommand{\radphi}{b_{\phi}}
\newcommand{\loss}{\ell}
\newcommand{\poploss}{L}

\newcommand{\scorefunc}{s}
\newcommand{\scoreval}{\scorefunc}
\newcommand{\scorerv}{S}
\newcommand{\quant}{\mathsf{Quant}}

\begin{document}

\title{A Few Observations on Sample-Conditional Coverage in
  \\ Conformal Prediction}
\author{John Duchi\thanks{Supported by the Office of Naval
    Research, grant N00014-22-1-2669, and Samsung grant SPO-313746.} \\
Stanford University}

\maketitle

\begin{abstract}
  We revisit the problem of constructing predictive confidence sets
  for which we wish to obtain some type of conditional
  validity.
  We provide new arguments showing how ``split conformal'' methods
  achieve near desired coverage levels with high probability, a
  guarantee conditional on the validation data rather than
  marginal over it.
  In addition, we directly consider (approximate) conditional coverage,
  where, e.g., conditional on a covariate $X$ belonging to some
  group of interest, we would like a guarantee that a predictive set
  covers the true outcome $Y$.
  We show that the natural method of performing quantile regression on a
  held-out (validation) dataset yields minimax optimal guarantees of
  coverage here.
  Complementing these positive results, we also provide experimental
  evidence that interesting work remains to be done to develop
  computationally efficient but valid predictive inference methods.
\end{abstract}

\section{Introduction and background}

In conformal
prediction~\cite{VovkGaSh05,LeiWa14,LeiGSRiTiWa18,BarberCaRaTi21a}, we wish
to perform predictive inference on the outcome $Y$ coming from pairs $(X, Y)
\in \mc{X} \times \mc{Y}$.
The basic approach yields confidence sets $C(x) \subset \mc{Y}$, where given
a sample $(X_i, Y_i)_{i = 1}^n$, an estimated confidence set
$\what{C}$ provides the (marginal) coverage
\begin{equation}
  \label{eqn:marginal-guarantee}
  \P\left(Y_{n + 1} \in \what{C}(X_{n + 1})\right) \ge 1 - \alpha.
\end{equation}
Typically, to do this, we assume the existence of a scoring function
$\scorefunc : \mc{X} \times \mc{Y} \to \R$ and define confidence sets
of the form $C_\tau(x) \defeq \{y \mid \scorefunc(x, y) \le \tau\}$.
For example, when predicting $Y \in \R$ in regression, given a
predictor $f : \mc{X} \to \R$ the absolute error $\scorefunc(x, y) =
|f(x) - y|$ yields the familiar confidence set $C_\tau(x) = \{y \in \R \mid
|y - f(x)| \le \tau\} = [f(x) - \tau, f(x) + \tau]$ of values $y$ near
$f(x)$.

The classical (split-conformal) approach~\cite{VovkGaSh05,BarberCaRaTi21a}
uses the sample to find the threshold $\what{\tau}$ larger than the
observed scores on about $1 - \alpha$ fraction of the data, then notes that
$\scoreval(X_{n + 1}, Y_{n + 1})$ is likely to be smaller than this
threshold.
More formally, if $(X_i, Y_i)$ are exchangeable and we let $\scorerv_i =
\scoreval(X_i, Y_i)$, then for the order statistics $\scorerv_{(1)} \le
\scorerv_{(2)} \le \cdots \le \scorerv_{(n + 1)}$, we have
\begin{equation*}
  \P\left(\scorerv_{n + 1} > \scorerv_{(\ceil{(1 - \alpha) (n + 1)})}\right)
  \le \alpha,
\end{equation*}
because the probability that $S_{n + 1}$ is in the $\alpha$-largest fraction
of the observed scores is at most $\alpha$.
Then a bit of
bookkeeping~\cite[e.g.][Lemma 2]{RomanoPaCa19} shows that the
slightly enlarged empirical quantile
\begin{equation*}
  \what{\tau} \defeq \quant_{(1 - \alpha)(1 + 1/n)}(\scorerv_1, \ldots,
  \scorerv_n),
\end{equation*}
provides the guarantee
\begin{equation*}
  \P\left(\scorerv_{n + 1} > \what{\tau}\right) \le \alpha.
\end{equation*}
Written differently, the confidence set
\begin{equation*}
  \what{C}(x) \defeq \left\{y \in \mc{Y}
  \mid \scorefunc(x, y) \le \what{\tau}\right\}
\end{equation*}
satisfies
\begin{equation*}
  \P(Y_{n + 1} \in \what{C}(X_{n + 1}))
  = \P\left(\scorefunc(X_{n + 1}, Y_{n + 1}) \le \what{\tau}\right)
  = \P\left(\scorerv_{n + 1} \le \what{\tau}\right) \ge 1 - \alpha.
\end{equation*}

\subsection{On $X$-conditional coverage}
\label{sec:intro-x-conditional}

We might ask for more than the marginal
guarantee~\eqref{eqn:marginal-guarantee} in a few ways.
The first is to target $X$-\emph{conditional coverage}, where
given a new datapoint $X_{n+1}$, we wish to achieve
\begin{equation*}
  \P(Y_{n + 1} \in \what{C}(X_{n + 1}) \mid X_{n + 1}) \ge 1 - \alpha.
\end{equation*}
But there are fundamental challenges here.
Let us say a set valued mapping $\what{C}_n : \mc{X} \toto \mc{Y}$
achieves distribution-free conditional $(1 - \alpha)$ coverage if
for any $P$, when $(X_i, Y_i) \simiid P$ and $\what{C}_n$ is a function
of $(X_i, Y_i)_{i = 1}^n$, then
for $P$-almost-all $x$,
\begin{equation}
  \label{eqn:x-conditional-coverage}
  \P(Y_{n + 1} \in \what{C}_n(X_{n + 1}) \mid X_{n + 1} = x) \ge 1 - \alpha.
\end{equation}

Unfortunately, such a distribution-free guarantee is impossible.
Focusing on the case that $\mc{Y} = \R$ for simplicity, so that confidence
sets $\what{C}(x) \subset \R$, \citet[Proposition 4]{Vovk12} shows that
the Lebesgue measure $\mathsf{Leb}(\what{C}(x))$ is almost always infinite
(see also \citet{BarberCaRaTi21a}):
\begin{corollary}
  Let $\mc{X}$ be a metric space and assume that $X \in \mc{X}$ has
  continuous distribution.
  If $\what{C}$ provides distribution free $(1 - \alpha)$ conditional
  coverage, then for $P$-almost all
  $x \in \mc{X}$,
  \begin{equation*}
    \P(\mathsf{Leb}(\what{C}(x)) = + \infty) \ge 1 - \alpha.
  \end{equation*}
\end{corollary}
\noindent
Similar results apply when the marginals over $X$ are not
too discrete~\cite[Corollary 7.1]{DuchiGuJiSu24}.

These failures motivate relaxing the conditional coverage
condition~\eqref{eqn:x-conditional-coverage}.
Perhaps the simplest approach considers
\emph{group-conditional coverage}, where for groups
$G \subset \mc{X}$, one targets the guarantee
\begin{align}
  \label{eqn:group-conditional}
  \P(Y_{n + 1} \in \what{C}(X_{n + 1}) \mid X_{n + 1} \in G) \ge 1 - \alpha.
\end{align}
\citet[Sec.~4]{BarberCaRaTi21a} show how to achieve the
coverage~\eqref{eqn:group-conditional} by considering worst-case coverage
over all groups $G$ of interest; \citet{JungNoRaRo23} consider
variations on this guarantee. 

\citet*{GibbsChCa23} develop a compelling
relaxation extending this idea.
To begin, note that conditional coverage $\P(Y \in \what{C}(x) \mid X = x) =
1 - \alpha$ is equivalent to the condition
\begin{equation}
  \label{eqn:weighted-coverage}
  \E\left[w(X) \left(\indic{Y \in \what{C}(X)} - (1 - \alpha)\right)\right]
  = 0
\end{equation}
for all bounded (measurable) functions $w : \mc{X} \to \R$. Indeed, the
tower property of conditional expectations shows that
equality~\eqref{eqn:weighted-coverage} holds if and only if
\begin{equation*}
  \E\left[w(X) \left(\P(Y \in \what{C}(X) \mid X) - (1 - \alpha)\right)
    \right] = 0,
  ~~ \mbox{i.e.} ~~
  \E\left[\left|\P(Y \in \what{C}(X) \mid X) - (1 - \alpha)\right|
    \right] = 0,
\end{equation*}
where we take $w(x) = \sign(\P(Y \in \what{C}(X) \mid X = x) - (1 -
\alpha))$, so $\P(Y \in \what{C}(X) \mid X) = 1 - \alpha$ with
probability $1$.
Similarly, we obtain the one-sided
inequality~\eqref{eqn:x-conditional-coverage} if and only if
\begin{equation*}
  \E\left[w(X) \indic{Y \in \what{C}(X)}\right] \ge (1 - \alpha)
  \E[w(X)]
\end{equation*}
for all nonnegative bounded $w$.
Taking $w(x) = \indic{x \in G}$ for groups $G \subset \mc{X}$
shows that this guarantee implies the group-conditional
coverage~\eqref{eqn:group-conditional};
relaxing the
condition~\eqref{eqn:weighted-coverage} by considering subclasses of
weighting functions $\mc{W} \subset \{\mc{X} \to \R\}$
leads to the following definition~\cite{GibbsChCa23}:
\begin{definition}
  \label{definition:weighted-coverage}
  A confidence set $C : \mc{X} \toto \mc{Y}$ achieves
  \emph{$\mc{W}$-weighted $((1 - \alpha), \epsilon)$ coverage} if
  \begin{equation*}
    \big|
    \E\left[w(X) \left(\indic{Y \in C(X)} - (1 - \alpha)\right)\right]
    \big| \le \epsilon
  \end{equation*}
  for all $w \in \mc{W}$.
\end{definition}
\noindent
In Definition~\ref{definition:weighted-coverage}, we take the
confidence set mapping $C$ to be fixed; \citeauthor{GibbsChCa23}
allow $\what{C}$ to be random, in which case (in their definition)
the expectation is taken also over the $\what{C}$ itself.
Because we study coverage conditional on a sample
used to construct $\what{C}$, we maintain
Definition~\ref{definition:weighted-coverage}.

\citeauthor{GibbsChCa23}'s main two examples are cases in which $\mc{W}$
corresponds to a vector space of the form $\mc{W} = \{w \mid w(x) = \<v,
\phi(x)\>\}$ for some feature mapping $\phi : \mc{X} \to \R^d$, and when
$\mc{W}$ corresponds to a reproducing kernel Hilbert space.
At prediction time---on a new example $X_{n + 1}$---they perform
\emph{full conformal inference}~\cite{VovkGaSh05},
where implicitly for each $t \in \R$, they solve
\begin{equation*}
  \what{h}_{n+1,t}
  = \argmin_{h \in \mc{W}}
  \sum_{i = 1}^n \loss_\alpha(h(X_i) - \scorerv_i)
  + \loss_\alpha(h(X_{n+1}) - t)
\end{equation*}
for the quantile loss $\loss_\alpha(t) = \alpha \hinge{t}
+ (1 - \alpha) \hinge{-t}$,
then define the implicit confidence set
\begin{equation}
  \what{C}_{n}(X_{n+1}) \defeq \left\{y \in \mc{Y}
  \mid \scoreval(X_{n + 1}, y) \le \what{h}_{n + 1, \scoreval(X_{n + 1},
    y)}(X_{n + 1}) \right\}.
  \label{eqn:implicit-full-confidence-set}
\end{equation}
A careful duality calculation~\cite[Sec.~4]{GibbsChCa23}
shows how to compute $\what{C}_n$ by solving a linear program
over $O(n + d)$ variables using $(X_i)_{i = 1}^{n + 1}$
and $\scorerv_1^n$,
and \citeauthor{GibbsChCa23}
show the set~\eqref{eqn:implicit-full-confidence-set}
satisfies
\begin{equation*}
  \left|\E\left[w(X_{n+1})\indic{
      Y_{n + 1} \not \in \what{C}_n(X_{n+1})}
    - w(X_{n + 1}) (1 - \alpha)\right]\right|
  \le \epsilon_{\textup{int}}(w)
\end{equation*}
for all functions $w \in \mc{W}$, where $\epsilon_{\textup{int}}(w)$ is an
interpolation error term (which they control and is of typically
small order in $n$).
Defining $\P_w(A) = \E_P[w(X)
  \indic{A}] / \E_P[w(X)]$ to be the $w$-weighted probability of an event
$A$ for $w \ge 0$, this inequality demonstrates that
\begin{equation}
  \label{eqn:cherian-gibbs-guarantee}
  \P_w(Y_{n + 1} \in \what{C}(X_{n + 1})) \ge 1 - \alpha
  - \epsilon_{\textup{int}}(w)
\end{equation}
for all $w \ge 0, w \in \mc{W}$,
a more nuanced guarantee than the marginal
coverage~\eqref{eqn:marginal-guarantee}.

One challenge with this full conformal approach is that computing
the prediction set $\what{C}_n$ requires solving a sometimes costly
optimization.
This suggests split-conformal approaches that provide adaptive
confidence sets of the form
\begin{equation*}
  \what{C}_n(x) \defeq \left\{y \in \mc{Y} \mid \scorefunc(x, y)
  \le \what{h}_n(x)\right\},
\end{equation*}
where $\what{h}_n$ is chosen based only on the sample $(X_i, Y_i)_{i =
  1}^n$---hence the name ``split conformal''---making the set $\what{C}_n$
easy to compute~\cite{RomanoPaCa19, CauchoisGuDu21}.
In spite of their ease of computation, it has been challenging to
demonstrate that these sets can achieve coverage;
for example, \citet{RomanoPaCa19} and \citet{CauchoisGuDu21}
apply another level of conformalization to
fit a constant threshold $\what{\tau}_n$ and use
$\what{C}_n(x) = \{y \in \mc{Y} \mid \scorefunc(x, y) \le
\what{h}_n(x) + \what{\tau}_n\}$.
%
%
We revisit these types of sets and show a few new convergence
and coverage guarantees for them, including new optimality guarantees.

\subsection{Sample-conditional coverage}
\label{sec:intro-sample-conditional}

Inequalities~\eqref{eqn:marginal-guarantee}
and~\eqref{eqn:cherian-gibbs-guarantee} provide
guarantees marginal over the entire sampling
procedure, drawing $(X_i, Y_i) \simiid P$ for $i = 1, 2, \ldots, n + 1$.
While conditional coverage (on $X_{n + 1}$) is
impossible, it is possible to achieve
\emph{sample-conditional}-coverage.
For this, let $P_n$ denote the empirical distribution of $(X_i, Y_i)_{i =
  1}^n$ (or simply the sample itself).
Then we seek a guarantee of the form
\begin{equation*}
  \P\left(Y_{n + 1} \in \what{C}_n(X_{n + 1}) \mid P_n\right)
  \ge 1 - \alpha - o(1)
\end{equation*}
with high probability over the sampling generating $\what{C}_n$, which is of
course a function of $P_n$.
Because of their reliance on individual examples, full-conformal procedures
cannot achieve such conditional guarantees~\cite{BianBa22}, though the
split-conformal procedures we review can straightforwardly guarantee them.
Our main purpose is to provide new arguments for this and
to extend the arguments to show how split-conformal methods can achieve
approximate weighted-conditional coverage
(Definition~\ref{definition:weighted-coverage}) with high probability
and optimal error $\epsilon$.
For example, recalling the weighted
probability~\eqref{eqn:cherian-gibbs-guarantee}, we will show that with high
probability over the sample $P_n$,
\begin{equation*}
  \P_w(Y_{n + 1} \in \what{C}_n(X_{n + 1}) \mid P_n) \ge
  1 - \alpha - O(1) \sqrt{\frac{\alpha(1 - \alpha)}{\E_P[w(X)]}
    \cdot \frac{d \log n}{n}}
\end{equation*}
simultaneously for all $w \ge 0$ in $d$-dimensional classes of functions
$\mc{W}$.

We recapitulate the basic arguments to achieve sample-conditional coverage
with naive split-conformal confidence sets $\what{C}_n(x) =
\{y \mid \scorefunc(x, y) \le \what{\tau}_n\}$ for a fixed threshold
$\what{\tau}_n$.
%
%
\citet[Section 3]{Vovk12} shows that the
event $\{Y_{n + 1} \not \in \what{C}_n(X_{n + 1})\}
= \{\scorefunc(X_{n + 1}, Y_{n + 1}) > \what{\tau}_n\}$
has small probability via an
argument working with individual failures $\scoreval(X_i, Y_i) > t$
for a particular threshold $t$.
We provide two related arguments to obtain sample-conditional coverage here.
The first uses elementary calculations and recapitulates Vovk's
argument for completeness but using our notation; the second uses a
concentration and VC-dimension calculation to preview our coming approaches.

To state things formally, let $\scorerv_i = \scoreval(X_i, Y_i)$
$i = 1, \ldots, n$, where $(X_i, Y_i) \simiid P$ for some
fixed distribution $P$. Let $\alpha \in (0, 1)$ be a desired
confidence level, and define the
empirical $(1 - \alpha)$ quantile
\begin{equation*}
  \what{\tau}_n \defeq \inf\left\{t \in \R \mid
  P_n(\scorerv \le t) \ge 1 - \alpha \right\},
\end{equation*}
where we recall that $P_n$ denotes the empirical distribution.
Then given this quantile, we define the confidence set
\begin{equation*}
  \what{C}_n(x) \defeq \left\{y \in \mc{Y} \mid \scorefunc(x, y)
  \le \what{\tau}_n \right\}.
\end{equation*}

\begin{proposition}[Vovk~\cite{Vovk12}, Proposition~2]
  \label{proposition:sample-conditional-basic}
  Let the construction above hold. Then for any $\gamma > 0$, with
  probability at least $1 - e^{-2n\gamma^2}$ over the sample $P_n$,
  \begin{equation}
    \P(Y_{n + 1} \in \what{C}_n(X_{n + 1}) \mid P_n) \ge 1 - \alpha - \gamma.
    \label{eqn:sample-conditional-basic}
  \end{equation}
\end{proposition}

\citet{JungNoRaRo23} consider this sample-conditional coverage with the
additional desideratum of
group-conditional-coverage~\eqref{eqn:group-conditional}, showing that under
some assumptions on the smoothness of the underlying distribution of
$\scoreval(x, Y)$ given $X = x$, it is approximately achievable.
We revisit their approach in Section~\ref{sec:approaching-conditional},
providing new and sharper guarantees on its behavior, without
any assumptions on the underlying distribution, allowing
analogues of the guarantee~\eqref{eqn:sample-conditional-basic}
in approximate conditional senses.
\citet{BianBa22} also consider such sample-conditional coverage, showing
that it is impossible to achieve without stronger assumptions for
many predictive methods. 

\section{Sample conditional coverage revisited}

We begin by revisiting the sample-conditional coverage
guarantees of Proposition~\ref{proposition:sample-conditional-basic}.
It admits a quite elementary proof relying only on Hoeffding's
concentration inequality, making it a natural point of
departure for developing more sophisticated coverage guarantees.
We thus provide this elementary proof, then
demonstrate the result using uniform convergence techniques.
These uniform convergence guarantees---which form the basis
for providing guarantees for approximate weighted coverage
(Definition~\ref{definition:weighted-coverage}) also
provide a two-sided bound on sample-conditional coverage:
\begin{corollary}
  \label{corollary:distinct-scores}
  Assume the scores $\scorerv_i = \scoreval(X_i, Y_i)$ are distinct with
  probability 1.
  Then for any $\gamma > 0$, with probability at least
  $1 - 2 e^{-2n \gamma^2}$ over the sample $P_n$,
  \begin{equation*}
    1 - \alpha - \gamma
    \le \P(Y_{n + 1} \in \what{C}_n(X_{n + 1}) \mid P_n)
    \le 1 - \alpha + \frac{1}{n} + \gamma.
  \end{equation*}
\end{corollary}

The simplicity of the guarantee~\eqref{eqn:sample-conditional-basic}
means it admits elegant extensions as well.
For example, we can extend the argument to give a bound that
more carefully tracks the desired confidence $\alpha$:

\begin{proposition}
  \label{proposition:distinct-scores}
  Let $\delta \in (0, 1)$ and define
  \begin{equation*}
    \gamma_n(\delta) \defeq \frac{4 \log \frac{1}{\delta}}{3n}
    + \sqrt{\Big(\frac{4}{3 n} \log \frac{1}{\delta}\Big)^2 +
      \frac{2 \alpha(1 - \alpha)}{n}
      \log \frac{1}{\delta}}
    \le \frac{8 \log \frac{1}{\delta}}{3n}
    + \sqrt{\frac{2 \alpha(1 - \alpha)}{n} \log \frac{1}{\delta}}.
  \end{equation*}
  Then with probability at least $1 - \delta$
  over the draw of the sample $P_n$,
  \begin{equation*}
    1 - \alpha - \gamma_n(\delta)
    \le \P(Y_{n + 1} \in \what{C}_n(X_{n+1}) \mid P_n).
  \end{equation*}
  If additionally the scores $\scorerv$ have a density,
  then with probability at least $1 - 2 \delta$,
  \begin{equation*}
    1 - \alpha - \gamma_n(\delta)
    \le \P(Y_{n + 1} \in \what{C}_n(X_{n+1}) \mid P_n)
    \le 1 - \alpha + \gamma_n(\delta).
  \end{equation*}
\end{proposition}

Roughly, we see that the simple quantile-based confidence set achieves
coverage
\begin{equation*}
  1 - \alpha \pm O_P(1) \sqrt{\frac{\alpha(1 - \alpha)}{n}}.
\end{equation*}
When $\alpha$ is small---which is the typical case---this is
always sharper than the naive guarantee~\eqref{eqn:sample-conditional-basic}.
The central limit theorem shows this is as accurately as we could hope to
even estimate the coverage level of a predictor; moreover, as we discuss
following Theorem~\ref{theorem:two-sided-sharp}, it is minimax
(rate) optimal.
In the remainder of the section, we provide two proofs
of Proposition~\ref{proposition:sample-conditional-basic},
along with Corollary~\ref{corollary:distinct-scores}.
In Section~\ref{sec:proof-distinct-scores}, we prove
Proposition~\ref{proposition:distinct-scores} using
Bernstein-type concentration guarantees.

\subsection{An elementary proof of
  Proposition~\ref{proposition:sample-conditional-basic}}

For the scalar random variable $\scorerv$, define the $\beta$-quantile
\begin{equation}
  \label{eqn:actual-quantile-mapping}
  q\opt(\beta) \defeq \inf \left\{t \in \R
  \mid \P(\scorerv \le t) \ge \beta \right\}.
\end{equation}
Because the CDF is right continuous, we have $\P(S \le q\opt(\beta)) \ge
\beta$, and $\P(S > q\opt(\beta)) = 1 - \P(S \le q\opt(\beta)) \le 1 -
\beta$.
For $\gamma > 0$ and any $\tau \in \R$, the inequality
\begin{equation*}
  \P(\scorerv_{n + 1} > \tau) > \alpha + \gamma,
  ~~ \mbox{i.e.} ~~
  \P(\scorerv_{n + 1} \le \tau) < 1 - \alpha - \gamma,
\end{equation*}
implies that $\tau < q\opt(1 - \alpha - \gamma)$.

Consider the event that $\what{\tau}_n < q\opt(1 - \alpha - \gamma)$.
For this to occur, it must be the case that
\begin{equation}
  \label{eqn:quantile-too-small}
  \frac{1}{n} \sum_{i = 1}^n \indic{\scorerv_i < q\opt(1 - \alpha - \gamma)}
  \ge 1 - \alpha.
\end{equation}
But this event is unlikely: define the Bernoulli indicator variables $B_i =
\indic{\scorerv_i < q\opt(1 - \alpha - \gamma)}$.
Then $\E[B_i] \le 1 -
\alpha - \gamma$, and Hoeffding's inequality implies
that $\wb{B}_n = \frac{1}{n} \sum_{i = 1}^n B_i$ satisfies
\begin{align*}
  \P\left(P_n(\scorerv < q\opt(1 - \alpha - \gamma)) \ge 1 - \alpha
  \right)
  & = \P\left(\wb{B}_n \ge 1 - \alpha\right) \\
  & \le \P\left(\wb{B}_n - \E[\wb{B}_n] \ge \gamma\right)
  \le \exp(-2 n \gamma^2).
\end{align*}
That is,
\begin{align*}
  \P\left(\what{\tau}_n < q\opt(1 - \alpha - \gamma)\right)
  \le \exp\left(-2n\gamma^2 \right)
\end{align*}
for any $\gamma > 0$, and so we must have
\begin{equation*}
  \P\left(\scorerv_{n + 1} > \what{\tau}_n \mid P_n \right) \le \alpha + \gamma
  ~~ \mbox{with~probability~at~least}~
  1 - e^{-2n\gamma^2}.
\end{equation*}
Rearranging and recalling that $Y_{n + 1} \not \in \what{C}_n(X_{n + 1})$ if
and only if $\scoreval(X_{n+1}, Y_{n + 1}) > \what{\tau}_n$, i.e., if
$\scorerv_{n + 1} > \what{\tau}_n$ gives the result.

\subsection{A proof of Proposition~\ref{proposition:sample-conditional-basic}
  using uniform convergence}
\label{sec:baby-uniform-convergence}

Our alternative approach to the proof of
Proposition~\ref{proposition:sample-conditional-basic} uses
the bounded differences inequality and a uniform concentration
guarantee.
First, for any estimated threshold $\what{\tau}_n$, we have
the trivial inequality
\begin{align*}
  \P(\scorerv_{n + 1} > \what{\tau}_n \mid P_n)
  & = \P(\scorerv_{n + 1} > \what{\tau}_n \mid P_n)
  - P_n(\scorerv > \what{\tau}_n)
  + P_n(\scorerv > \what{\tau}_n) \\
  & \le \sup_{\tau \in \R}
  \left|P(\scorerv > \tau)
  - P_n(\scorerv > \tau)\right|
  + P_n(\scorerv > \what{\tau}_n).
\end{align*}
Then because we choose $\what{\tau}_n$ so that
$P_n(\scorerv \le \what{\tau}_n) \ge 1 - \alpha$, we obtain
\begin{subequations}
  \label{eqn:one-VC-simple-deviation}
  \begin{equation}
    \P(\scorerv_{n + 1} > \what{\tau}_n \mid P_n)
    \le \alpha + \sup_{\tau \in \R}
    \left|P(\scorerv > \tau) - P_n(\scorerv > \tau)\right|.
  \end{equation}
  If the values $\scorerv_i$ are distinct, then $P_n(\scorerv \le
  \what{\tau}_n) \le 1 - \alpha + \frac{1}{n}$, and so a completely similar
  calculation yields
  \begin{equation}
    \P(\scorerv_{n + 1} > \what{\tau}_n \mid P_n)
    \ge \alpha - \frac{1}{n} - \sup_{\tau \in \R}
    \left|P(\scorerv > \tau) - P_n(\scorerv > \tau)\right|.
  \end{equation}
\end{subequations}
In either case, if we can control the deviation
$|P(\scorerv > \tau) - P_n(\scorerv > \tau)|$ uniformly across
$\tau$, we will have evidently proved the desired result.

We consider two arguments, the first yielding sharper constants,
while the second generalizes to weighted
coverage.
For the first, we apply the Dvoretsky-Kiefer-Wolfowitz
inequality~\cite{Massart90}:
\begin{equation*}
  \P\left(\sup_{\tau \in \R} |P(\scorerv > \tau) - P_n(\scorerv > \tau)|
  \ge t \right) \le 2 e^{-2n t^2}
\end{equation*}
for all $t \ge 0$.
Combining the equations~\eqref{eqn:one-VC-simple-deviation}, we
thus obtain that
\begin{align*}
  \P(\scorerv_{n + 1} > \what{\tau}_n \mid P_n)
  \le \alpha + \gamma
  ~~ \mbox{with probability at least}~
  1 - 2 e^{-2 n \gamma^2}.
\end{align*}
If the scores are distinct, the corresponding lower
bound is immediate, giving Corollary~\ref{corollary:distinct-scores}.

The final alternative argument to control the uniform deviations
in the bounds~\eqref{eqn:one-VC-simple-deviation} underpins
our more sophisticated guarantees in the sequel,
relying on uniform concentration guarantees and the
Vapnik-Chervonenkis (VC) dimension.
First, recall the classical bounded differences
inequality~\cite{McDiarmid89, Wainwright19}, where we say a function $f :
\mc{X}^n \to \R$ satisfies $c_i$-bounded differences if
\begin{equation*}
  |f(x_1^{i-1}, x_i, x_i, x_{i + 1}^n)
  - f(x_1^{i-1}, x_i', x_{i+1}^n)|
  \le c_i
  ~~ \mbox{for~all~} x_1^{i-1}, x_{i+1}^n,
  x_i, x_i' \in \mc{X}.
\end{equation*}
\begin{lemma}[Bounded differences]
  \label{lemma:bounded-differences}
  Let $X_1, \ldots, X_n$ be independent random variables
  and $f$ satisfy $c_i$-bounded differences. Then
  for all $t \ge 0$,
  \begin{equation*}
    \max\left\{
    \P(f(X_1^n) - \E[f(X_1^n)] \ge t),
    \P(f(X_1^n) - \E[f(X_1^n)] \le -t)
    \right\}
    \le \exp\left(-\frac{2 t^2}{\sum_{i = 1}^n c_i^2}\right).
  \end{equation*}
\end{lemma}

We then observe that $f(P_n) \defeq \sup_{\tau \in \R} |P(\scorerv > \tau) -
P_n(\scorerv > \tau)|$ trivially satisfies bounded differences.
Indeed, let $P_n'$ differ from $P_n$ in a single observation.
Then
defining $\linf{P - P_n} = \sup_\tau |P(\scorerv > \tau)
- P_n(\scorerv > \tau)|$ for shorthand, we obtain
\begin{align*}
  \left|\linf{P - P_n} - \linf{P - P_n'}\right|
  & \le
  \linf{P_n - P_n'}
  \le \frac{1}{n}
\end{align*}
by the triangle inequality and that only one example may change.
Lemma~\ref{lemma:bounded-differences} then implies
\begin{align*}
  \P\left(\linf{P - P_n} \ge \E[\linf{P - P_n}]
  + t \right) \le e^{-2nt^2}
\end{align*}
for $t \ge 0$, so that we need only control $\E[\linf{P - P_n}]$.
For this, we perform a standard symmetrization
argument~\cite[e.g.][Ch.~2.3]{VanDerVaartWe96}: let $P_n^0 = \frac{1}{n}
\sum_{i = 1}^n \varepsilon_i \pointmass_{\scorerv_i}$, where $\varepsilon_i
\simiid \uniform\{\pm 1\}$ are i.i.d.\ Rademacher variables
and $\pointmass_{\scorerv_i}$ denotes a point mass on $\scorerv_i$.
By introducing independent copies of $\scorerv_i$
and applying Jensen's inequality~\cite[Lemma 2.3.1]{VanDerVaartWe96},
we have the bound
\begin{equation*}
  \E\left[\linf{P_n - P}\right]
  \le 2 \E\left[\linf{P_n^0}\right]
  = 2 \E\left[\sup_{\tau \in \R} \bigg|\frac{1}{n}
    \sum_{i = 1}^n \varepsilon_i \indic{\scorerv_i > \tau}
    \bigg|\right].
\end{equation*}
Because the class of functions $\scoreval \mapsto \indic{\scoreval > \tau}$
has VC-dimension at most 1, Dudley's entropy
integral (see, e.g.~\cite[Corollary 2.2.8 and Thm.~2.6.7]{VanDerVaartWe96}
or~\cite[Eq.~(5.5.1)]{Wainwright19}) shows that
\begin{equation*}
  \E\left[\linf{P_n^0}\right]
  \le \frac{c}{\sqrt{n}}
\end{equation*}
for a numerical constant $c$.
We then obtain that for any $\gamma \ge 0$,
\begin{align*}
  \P\left(\scorerv_{n + 1} > \what{\tau}_n \mid P_n\right)
  \le \alpha + \frac{c}{\sqrt{n}} + \gamma
  ~~ \mbox{w.p.}~
  1 - e^{-2n \gamma^2}
\end{align*}
by the inequalities~\eqref{eqn:one-VC-simple-deviation}, where
$c$ is a numerical constant.

\subsection{Proof of Proposition~\ref{proposition:distinct-scores}}
\label{sec:proof-distinct-scores}


Recall the quantile mapping $q\opt$ from
the definition~\eqref{eqn:actual-quantile-mapping} and that
for fixed $\gamma \in [0, \alpha]$, the event
$\what{\tau}_n < q\opt(1 - \alpha - \gamma)$
can occur only if
$P_n(\scorerv < q\opt(1 - \alpha - \gamma)) \ge 1 - \alpha$.
Then defining $B_i = \indic{\scorerv_i < q\opt(1 - \alpha - \gamma)}$
and recalling inequality~\eqref{eqn:quantile-too-small},
we obtain
\begin{align*}
  \P(\what{\tau}_n < q\opt(1 - \alpha - \gamma))
  \le \P(\wb{B}_n \ge 1 - \alpha)
  = \P(\wb{B}_n - \E[\wb{B}_n] \ge 1 - \alpha - \E[\wb{B}_n]).
\end{align*}
Define $p(\gamma) = \E[B_i]
= \P(\scorerv < q\opt(1 - \alpha - \gamma)) < 1 - \alpha - \gamma$,
so that $t = t(\gamma) \defeq 1 - \alpha - p(\gamma) > \gamma$.
Then $\var(B_i) = p(\gamma)(1 - p(\gamma))
= (\alpha + t)(1 - \alpha - t)$, and Bernstein's inequality
implies
\begin{align*}
  \P(\wb{B}_n - \E[\wb{B}_n] \ge t)
  & \le \exp\left(-\frac{n t^2}{2 (1 - \alpha - t) (\alpha + t)
    + \frac{2}{3} t}\right) \\
  & = \exp\left(-\frac{n t^2}{2 (1 - \alpha) \alpha +
    (\frac{8}{3} - 4 \alpha) t - t^2}\right)
  \le \exp\left(-\frac{n t^2}{2(1 - \alpha) \alpha
    + \frac{8}{3} t}\right).
\end{align*}
Notably, $t \mapsto \frac{nt^2}{2(1 - \alpha) \alpha + \frac{8}{3} t}$
is increasing in $t$, so that
\begin{equation*}
  \P(\what{\tau}_n < q\opt(1 - \alpha - \gamma))
  \le \exp\left(-\frac{n \gamma^2}{2(1 - \alpha) \alpha + \frac{8}{3} \gamma}
  \right).
\end{equation*}

If the scores $\scorerv$ have a density,
$\P(\scorerv \le q\opt(\beta)) = \beta$ for any $\beta \in (0, 1)$.
Then we may also consider the event that $\what{\tau}_n > q\opt(1 - \alpha +
\gamma)$.
For this to occur, we require
\begin{equation*}
  P_n(\scorerv < q\opt(1 - \alpha + \gamma)) \le 1 - \alpha,
\end{equation*}
and defining $B_i = \indic{\scorerv_i < q\opt(1 - \alpha + \gamma)}$,
we have $\E[B_i] = 1 - \alpha + \gamma$ and so
\begin{align*}
  \P(\wb{B}_n \le 1 - \alpha)
  = \P(\wb{B}_n - \E[\wb{B}_n] \le -\gamma)
  & \le \exp\left(-\frac{n \gamma^2}{2 (1 - \alpha + \gamma)
    (\alpha - \gamma) + \frac{2}{3} \gamma}\right) \\
  & \le \exp\left(-\frac{n \gamma^2}{2 (1 - \alpha)\alpha + \frac{2}{3}
    \gamma}\right)
\end{align*}
for $\gamma \in [0, \alpha]$.
Combining the two cases, for $\gamma \ge 0$ we have
\begin{equation*}
  \max\left\{\P\left(\what{\tau}_n < q\opt(1 - \alpha - \gamma)\right),
  \P\left(\what{\tau}_n > q\opt(1 + \alpha + \gamma)\right)\right\}
  \le \exp\left(-\frac{n \gamma^2}{2\alpha(1 - \alpha) + \frac{8}{3}
    \gamma}\right).
\end{equation*}
Solving to guarantee the right hand side is at most $\delta$ yields
\begin{equation*}
  \gamma_n \defeq \frac{4 \log \frac{1}{\delta}}{3n}
  + \sqrt{\left(\frac{4}{3 n} \log \frac{1}{\delta}\right)^2 +
    \frac{2 \alpha(1 - \alpha)}{n}
    \log \frac{1}{\delta}}.
\end{equation*}
Applying a union bound implies
Proposition~\ref{proposition:distinct-scores}.

\section{Approaching conditional coverage}
\label{sec:approaching-conditional}

Keeping in mind the ideas in Sections~\ref{sec:intro-x-conditional}
and~\ref{sec:intro-sample-conditional}, we revisit conditional coverage,
but we do so conditional on the sample $P_n$ as well.
To do this, we revisit Jung et al.\ and Gibbs et al.'s
approaches~\citep{JungNoRaRo23,GibbsChCa23}, considering
quantile estimation via the quantile loss~\citep{KoenkerBa78},
which for $\alpha > 0$ is
\begin{equation*}
  \loss_\alpha(t) \defeq \alpha \hinge{t} + (1 - \alpha) \hinge{-t}.
\end{equation*}
For a random variable $Y$, the $(1 - \alpha)$-quantile $\quant_{1 -
  \alpha}(Y) \defeq \inf\{t \mid \P(Y \le t) \ge 1 - \alpha\}$ minimizes
$\poploss(t) \defeq \E[\loss_\alpha(t - Y)]$.
This is easy to see when $Y$ has a density, in which case we can take
derivatives (ignoring the measure-0 set of points where $\loss_\alpha(t -
Y)$ is non-differentiable~\cite{Bertsekas73}) to obtain
\begin{equation*}
  \poploss'(t) = \alpha\P(t - Y > 0) - (1 - \alpha) \P(t - Y \le 0)
  = \P(Y < t) - (1 - \alpha) = 0
\end{equation*}
if and only if $\P(Y < t) = 1 - \alpha$; that is,
$\quant_{1 - \alpha}(Y)$ is always \emph{a} minimizer.
(When $Y$ has point masses, a minor extension of this
argument shows that $\quant_{1 - \alpha}(Y)$ remains a minimizer.)

Following \citet{GibbsChCa23} and \citet{JungNoRaRo23},
consider a multi-dimensional quantile regression
of attempting to predict $\scorerv \in \R$ from $X \in \mc{X}$.
Let $\phi : \mc{X} \to \R^d$ be a feature mapping, and consider the population
loss $\poploss(\theta) \defeq \E[\loss_\alpha(\<\theta, \phi(X)\> - \scorerv)]$.
Then (for motivation) assuming that $\scorerv$ has a density
conditional on $X = x$, we see that
\begin{align*}
  \nabla \poploss(\theta)
  & = \alpha \E\left[\indic{\<\theta, \phi(X)\> - \scorerv > 0}
    \phi(X)\right]
  - (1 - \alpha) \E\left[\indic{\<\theta, \phi(X)\> - \scorerv \le 0} \phi(X)
    \right] \\
  & = \E\left[\left(\indic{\scorerv < \<\theta, \phi(X)\>} - (1 - \alpha)\right)
    \phi(X) \right].
\end{align*}
Now let
\begin{equation*}
  \theta\opt \in \argmin \poploss(\theta)
\end{equation*}
be a population minimizer.
Then, as \citet{GibbsChCa23} note, for \emph{any} $u \in \R^d$, we have
\begin{equation*}
  0 = \<u, \nabla \poploss(\theta\opt)\>
  = \E\left[\left(\P(\scorerv < \<\theta\opt, \phi(X)\> \mid X)
    - (1 - \alpha)\right) \cdot \<u, \phi(X)\> \right].
\end{equation*}
This connects transparently to confidence set
mappings~\cite{GibbsChCa23,JungNoRaRo23}: taking
\begin{equation*}
  \scorerv = \scoreval(X, Y)
  ~~ \mbox{and} ~~
  C_{\theta\opt}(x) \defeq \left\{y \in \mc{Y} \mid \scoreval(x, y) \le
  \<\theta\opt, \phi(X)\>\right\},
\end{equation*}
we have
\begin{align*}
  0 = \E\left[\left(\P(Y \in C(X) \mid X) - (1 - \alpha)\right)
    \<u, \phi(X)\>\right] = 0
  ~~ \mbox{for~all~} u \in \R^d.
\end{align*}
In turn, this implies the population coverage guarantee
\begin{corollary}
  \label{corollary:population-weighted-coverage}
  Assume the distribution of $\scorerv$ conditional on $X = x$ has no atoms
  for each $x \in \mc{X}$. Let $\theta\opt$ minimize $\poploss(\theta) =
  \E[\loss_\alpha(\<\theta, \phi(X)\> - \scoreval(X, Y))]$.
  Then $C_{\theta\opt}$ provides $((1 - \alpha), 0)$-weighted coverage
  (Definition~\ref{definition:weighted-coverage}) for the class $\mc{W}
  \defeq \{w \mid w(x) = \<u, \phi(x)\>\}_{u \in \R^d}$ of linear functions
  of $\phi(x)$.
\end{corollary}

Two questions arise: first, whether the corollary
extends to $\scorerv$ which may have atoms.
A (trivial) workaround is to simply add a tiny amount of random noise to
each $\scorerv_i$.
Otherwise, it is generally possible only to provide a one-sided bound, where
we weight with nonnegative functions $w$ (we provide this in the sequel).
The more interesting question is how we can extend
these guarantees to provide sample-conditional coverage.
Adapting the arguments we use in Section~\ref{sec:baby-uniform-convergence}
to prove Proposition~\ref{proposition:sample-conditional-basic} and
Corollary~\ref{corollary:distinct-scores}
allows us to address this.


\subsection{An estimated confidence set}

The population-level confidence set $C_{\theta\opt}(x) = \{y \mid
\scoreval(x, y) \le \<\theta\opt, \phi(x)\>\}$ immediately suggests
developing an empirical analogue~\cite{GibbsChCa23, JungNoRaRo23}.
Thus, we turn to an analysis of the empirical confidence set,
considering the estimator
\begin{equation}
  \what{\theta}_n \in
  \argmin_\theta
  \E_{P_n}\left[\loss_\alpha(\<\theta, \phi(X)\> - \scorerv)\right],
  \label{eqn:empirical-quantile-estimator}
\end{equation}
which \citet{JungNoRaRo23} consider for the special
case that the feature mapping $\phi(x) = [\indic{x \in G}]_{G \in \mc{G}}$ is an
indicator vector for groups $G \subset \mc{X}$.
This gives
the confidence set
\begin{equation*}
  \what{C}_n(x) \defeq \left\{y \in \mc{Y}
  \mid \scoreval(x, y) \le \<\what{\theta}, \phi(x)\> \right\}.
\end{equation*}

We first sketch an argument that this set provides a desired approximate
weighted coverage.
By convexity,
\begin{align*}
  0 \in \sum_{i = 1}^n \partial \loss_\alpha\left(
  \<\what{\theta}, \phi(X_i)\>
  - \scorerv_i\right) \phi(X_i),
\end{align*}
which is equivalent to the statement that for
some scalars (really, dual variables) $\eta_i$ satisfying
\begin{equation}
  \label{eqn:g-gradient-calcs}
  \eta_i = \begin{cases}
    \alpha & \mbox{if}~ \<\what{\theta}, \phi(X_i)\> > \scorerv_i \\
    -(1 - \alpha) & \mbox{if}~ \<\what{\theta}, \phi(X_i)\> < \scorerv_i \\
    \in [-(1 - \alpha), \alpha]
    & \mbox{if}~ \<\what{\theta}, \phi(X_i)\> = \scorerv_i
  \end{cases}
\end{equation}
we have
\begin{align*}
  0 = \sum_{i = 1}^n \eta_i \phi(X_i).
\end{align*}
Let us proceed heuristically for now by taking ``discrete'' values for the
$\eta_i$.
If we assume that
\begin{equation*}
  \eta_i =
  \alpha \indics{\<\what{\theta}, \phi(X_i)\> > \scorerv_i}
  - (1 - \alpha) \indics{\<\what{\theta}, \phi(X_i)\> \le \scorerv_i}
  = \indics{\scorerv_i < \<\what{\theta}, \phi(X_i)\>}
  - (1 - \alpha),
\end{equation*}
then
\begin{equation*}
  0 = \sum_{i = 1}^n
  \phi(X_i) \left(\indic{\scorerv_i \le \<\what{\theta}, \phi(X_i)\>}
  -  (1 - \alpha)\right)
  = \sum_{i = 1}^n
  \phi(X_i) \left(\indic{Y_i \in \what{C}_n(X_i)}
  -  (1 - \alpha)\right),
\end{equation*}
the empirical version of
Corollary~\ref{corollary:population-weighted-coverage}.
If we could argue that this (heuristic)
empirical average concentrates around its expectation,
we would obtain
\begin{equation*}
  \E\left[\phi(X_{n + 1}) \left(\indic{Y_{n + 1} \not\in \what{C}_n(X_{n+1})}
    - \alpha\right) \mid P_n\right]
  \approx \frac{1}{n}
  \sum_{i = 1}^n \phi(X_i)
  \left(\indic{\scorerv_i > \<\what{\theta}, \phi(X_i)\>}
  -  \alpha\right)
  = 0,
\end{equation*}
weighted coverage for a new example $(X_{n + 1}, Y_{n + 1})$.

\newcommand{\ballphi}{\ball_\phi}

To make this argument rigorous, we employ a bit of empirical
process theory and measure concentration.
For now and for simplicity, we assume that $\phi(x)$ is bounded in
$\ell_2$, so that $\ltwo{\phi(x)} \le \radphi$ for all $x$,
and let $\ball_2 = \{u \in \R^d \mid \ltwo{u} \le 1\}$ be the
$\ell_2$-ball.

\begin{theorem}
  \label{theorem:approximate-conditional-coverage}
  Assume the boundedness conditions above and that $n \ge d$.
  Let $\what{\theta}$ be the empirical
  minimizer~\eqref{eqn:empirical-quantile-estimator} of the
  $\alpha$-quantile loss, let $\what{h}(x) = \<\what{\theta}, \phi(x)\>$,
  and define the confidence set
  \begin{equation*}
    \what{C}_n(x) \defeq \left\{y \in \mc{Y}
    \mid \scoreval(x, y) \le \what{h}(x) \right\}.
  \end{equation*}
  Then there exists a constant
  $c \le 2 + \alpha/\sqrt{d}$ such that for $t \ge 0$,
  with probability at least $1 - e^{-n t^2}$ over the draw
  of the sample $P_n$,
  \begin{equation*}
    \E\left[\<u, \phi(X_{n + 1})\>
      \left(\indic{Y_{n + 1} \in \what{C}_n(X_{n+1})}
      - (1 - \alpha)\right) \mid P_n \right]
    \ge -c \radphi \left(\sqrt{\frac{d}{n}  \log \frac{n}{d}}
    + t \right)
  \end{equation*}
  simultaneously for all $u \in \ball_2$
  satisfying $\<u, \phi(x)\> \ge 0$ for all $x \in \mc{X}$.
  
  If additionally the scores $\scorerv_i$ are distinct with probability $1$,
  then with the same probability,
  \begin{equation*}
    \E\left[\<u, \phi(X_{n + 1})\>
      \left(\indic{Y_{n + 1} \in \what{C}_n(X_{n+1})}
      - (1 - \alpha)\right) \mid P_n \right]
    \le 3 \radphi \left(\sqrt{\frac{d}{n} \log \frac{n}{d}}
    + t + \frac{d}{3n} \right).
  \end{equation*}
  simultaneously for all $u \in \ball_2$.
\end{theorem}
\noindent
We defer the proof of Theorem~\ref{theorem:approximate-conditional-coverage}
to Section~\ref{sec:proof-approximate-conditional-coverage}.

Translating things to guarantee (nearly) exact coverage and distinct scores
even when the naive scores may not be distinct, we can
randomize, though we will assume the randomization scale makes it unlikely
to modify results (e.g., perturbing by noise at a scale of $10^{-10}$).
In this case, let $U_i$, $i = 1, 2, \ldots, n + 1$ be i.i.d.\ random
variables with a density on $\R$, and define the scores $\scorerv_i =
\scoreval(X_i, Y_i)$ as usual and perturbed scores $\scorerv_i + U_i$; then
(abusing notation) let $\what{\theta}_n$ be an empirical estimator
\begin{equation*}
  \what{\theta}_n \in \argmin_{\theta} \E_{P_n}\left[\loss_\alpha\left(
    \<\theta, \phi(X)\> - \scorerv - U \right)\right].
\end{equation*}
We obtain the following corollary.
\begin{corollary}
  Let $\what{\theta}_n$ be as above and define
  $\what{h}(x) = \<\what{\theta}_n, \phi(x)\>$. Define the
  confidence set
  \begin{equation*}
    \what{C}_n(x, u) \defeq \left\{y \in \mc{Y} \mid \scoreval(x, y)
    \le \what{h}(x) + u \right\}.
  \end{equation*}
  Then there exists a numerical constant $c \le 3$
  such that for
  $t \ge 0$, with probability at least $1 - e^{-n t^2}$,
  the randomized confidence set $\what{C}_n(\cdot, U)$ achieves
  $((1 - \alpha), \epsilon_n)$-weighted coverage
  (Definition~\ref{definition:weighted-coverage})
  for the class $\mc{W} \defeq \{w(x) = \<u, \phi(x)\>\}_{u \in \ball_2}$
  with
  \begin{equation*}
    \epsilon_n \le c \radphi \left(
    \sqrt{\frac{d \log \frac{n}{d}}{n}}
    + \frac{d}{n} + t \right).
  \end{equation*}
\end{corollary}

As another corollary to
Theorem~\ref{theorem:approximate-conditional-coverage}, let us assume that
$\mc{G} = \{G_1, \ldots, G_d\}$ consists of sets $G_i$ partitioning
$\mc{X}$, and define the
feature indicator $\phi_\mc{G}(x) = (1, \indic{x \in G_1}, \ldots, \indic{x \in
  G_d})$. With
this particular choice, we obtain the following result:
\begin{corollary}
  \label{corollary:group-conditional-baby}
  Let $\what{\theta}$ be as in
  Theorem~\ref{theorem:approximate-conditional-coverage}
  and $\phi = \phi_{\mc{G}}$ be the group feature function.
  Then simultaneously for all groups $G_j$,
  \begin{equation*}
    \P(Y_{n + 1} \in \what{C}_n(X_{n + 1}) \mid X_{n + 1} \in G_j, P_n)
    \ge 1 - \alpha - \frac{4}{\P(X_{n + 1} \in G_j)}
    \left(\sqrt{\frac{d}{n} \log \frac{n}{d}} + t \right)
  \end{equation*}
  with probability at least $1 - e^{-nt^2}$.
\end{corollary}
\noindent
We will sharpen this inequality via more sophisticated
arguments in the sequel.

It is instructive here, however,
to compare this guarantee to
that the
full-conformal approach~\eqref{eqn:implicit-full-confidence-set},
which depends on $X_{n+1}$, provides.
%
The construction~\eqref{eqn:implicit-full-confidence-set} appears
appears to obtain better coverage than the more basic approaches
here~\cite[Fig.~3]{GibbsChCa23}, but it can be more
computationally challenging~\cite[Fig.~6]{GibbsChCa23}.
Indeed, \citet{GibbsChCa23} show that a randomized version of their
procedure~\eqref{eqn:implicit-full-confidence-set} achieves
\begin{equation*}
  \P(Y_{n + 1} \in \what{C}_n(X_{n+1}) \mid X_{n + 1} \in G)
  = 1 - \alpha ~~ \mbox{for~all~} G \in \mc{G}.
\end{equation*}
This can be substantially sharper than the guarantee
Corollary~\ref{corollary:group-conditional-baby} provides, as our
sample-conditional coverage guarantees are not quite so exact; we revisit
these points in experiments.

\subsection{Proof of Theorem~\ref{theorem:approximate-conditional-coverage}}
\label{sec:proof-approximate-conditional-coverage}

We use the empirical process notation
$P_n f = \frac{1}{n} \sum_{i = 1}^n f(X_i)$ for shorthand.
Recall that for a convex function $f$, the directional
derivative $f'(x; u) = \lim_{t \downarrow 0} \frac{f(x + tu) - f(x)}{t}$
exists and satisfies $f'(x; u) = \sup\{\<g, u\> \mid g \in \partial f(x)\}$.
Thus, by definition of optimality,
\begin{align*}
  P_n \loss_\alpha'(\<\what{\theta}, \phi(X)\> - \scorerv; u) \ge 0
\end{align*}
for all $u$.
Let $u$ be such that $\<u, \phi(x)\> \ge 0$ for all $x \in \mc{X}$. Then
\begin{equation*}
  \loss_\alpha'(\<\theta, \phi(x)\> - \scoreval; u)
  = \<\phi(x), u\>
  \left[\alpha \indic{\<\theta, \phi(x)\> \ge \scoreval}
    - (1 - \alpha) \indic{\<\theta, \phi(x)\> < \scoreval}\right].
\end{equation*}
Then by the first-order optimality condition we obtain
\begin{align*}
  0 & \le \left\<u, \alpha P_n \phi(X) \indic{\<\what{\theta}, \phi(X)\>
    \ge \scorerv} - (1 - \alpha) P_n\phi(X)
  \indic{\<\what{\theta}, \phi(X)\> < \scorerv}\right\> \\
  & = \left\<u, \alpha P_n \phi(X)
  - P_n \phi(X) \indic{\<\what{\theta}, \phi(X)\>
    < \scorerv} \right\>.
\end{align*}
Suppose that we demonstrate that
\begin{equation*}
  \ltwobigg{\frac{1}{n} \sum_{i = 1}^n \phi(X_i) \indic{\scorerv_i >
      \<\theta, \phi(X_i)\>}
    -
  \E_P[\phi(X) \indic{\scorerv > \<\theta, \phi(X)\>}]} \le \epsilon
\end{equation*}
uniformly over $\theta \in \R^d$.
Then we would obtain
\begin{equation*}
  0 \le \left\<u, P_n \phi(X) \left(\alpha - \indic{\scorerv >
    \<\what{\theta}, \phi(X)\>}\right)\right\>
  \le \E_P\left[\<u, \phi(X)\> \left(\alpha - \indic{\scorerv
      > \<\what{\theta}, \phi(X)\>}\right)\right]
  + \epsilon,
\end{equation*}
for all $u \in \ball_2$ with $\<u, \phi(x)\> \ge 0$ for all $x \in \mc{X}$,
that is,
\begin{equation}
  \label{eqn:one-sided-thing-to-prove}
  \E_P\left[\<u, \phi(X)\> \left(
    \indic{Y \not \in \what{C}(X)} - \alpha\right)\right]
  \le \epsilon,
\end{equation}
as $y \not \in \what{C}(x)$ if and only if
$\scoreval(x, y) > \<\what{\theta}, \phi(x)\>$.
With appropriate $\epsilon$, this will give the first claim of the theorem.

To that end, we abstract a bit and let $\mc{H} \subset \{\mc{X} \to \R\}$ be
(for now) any collection of functions, and consider the process defined by
\begin{equation*}
  Z_n(h) \defeq \frac{1}{n}
  \sum_{i = 1}^n \phi(X_i) \indic{\scorerv_i > h(X_i)}.
\end{equation*}
When $\mc{H}$ is a VC-class, for each coordinate $j$, functions of the form
$\phi_j(x) \indic{\scoreval > h(x)}$ are VC-subgraph~\cite[Lemma
  2.6.18]{VanDerVaartWe96}.
So (at least abstractly) we expect $Z_n$ to concentrate around its
expectations at a reasonable rate.
The following technical lemma,
whose proof we provide
in Appendix~\ref{sec:proof-bound-deviation-ltwo}, provides one variant of this.
\begin{lemma}
  \label{lemma:bound-deviation-ltwo}
  Let $\ball_2 = \{u : \ltwo{u} \le 1\}$
  and $\mc{H}$ have VC-dimension $k$.
  Then
  \begin{equation*}
    \E\left[\sup_{h \in \mc{H}, u \in \ball_2}
      \<u, Z_n(h) - \E[Z_n(h)]\>\right]
    \le 2 \sqrt{\frac{k \log \frac{n e}{k}}{n}}
    \E\bigg[\frac{1}{n} \sum_{i = 1}^n \norm{\phi(X_i)}^2\bigg]^{1/2}.
  \end{equation*}
\end{lemma}
\noindent
The trivial inequality $\E[\sup_{u \in \ballphi} \<u, P_n \phi(X) -
  \E[\phi(X)]\>] \le \frac{1}{\sqrt{n}} \E[\ltwo{\phi(X)}^2]^{1/2}$
addresses terms involving $P_n \phi(X) \alpha$ above.

We can extend the lemma by homogeneity to capture arbitrary vectors.
To do so, note that
if we change a single example $(X_i, \scorerv_i)$, then
$\<u, Z_n(h)\>$ changes by at most
$n^{-1} \sup_x \<u, \phi(x)\> \le n^{-1} \ltwo{u} \sup_x \ltwo{\phi(x)}$.
Using homogeneity, for any scalar $t$ there exists $u \in \R^d$ such that
$\<u, Z_n(h) - \E[Z_n(h)]\> \ge \ltwo{u} t$ if and only if there exists $u
\in \sphere^{d-1}$ such that $\<u, Z_n(h) - \E[Z_n(h)]\> \ge t$.
So if $\radphi = \sup_{x \in \mc{X}} \ltwo{\phi(x)}$, we obtain by bounded
differences (Lemma~\ref{lemma:bounded-differences}) and homogeneity that
\begin{equation*}
  \P\left(\sup_{u \neq 0, h \in \mc{H}}
  \frac{\<u, Z_n(h) - \E[Z_n(h)]\>}{\ltwo{u}}
  \ge \radphi t
  +
  \E\left[\sup_{u \in \sphere^{d-1}, h \in \mc{H}}
    \<u, Z_n(h) - \E[Z_n(h)]\>\right]
  \right)
  \le e^{-n t^2}.
\end{equation*}
Summarizing, we have proved the following proposition.
\begin{proposition}
  \label{proposition:high-prob-deviation}
  Let $\mc{H}$ have VC-dimension $k$. Then for $t \ge 0$,
  \begin{equation*}
    \P\left(\sup_{u \neq 0, h \in \mc{H}}
    \frac{\<u, Z_n(h) - \E[Z_n(h)]\>}{\ltwo{u}}
    \ge 
    2 \radphi \sqrt{\frac{k \log \frac{n}{k}}{n}} + \radphi t \right)
    \le e^{-n t^2}.
  \end{equation*}
\end{proposition}

By taking $\mc{H} = \{h : h(x) = \<\theta, \phi(x)\>\}_{\theta \in \R^d}$,
which has VC-dimension $d$, in
Proposition~\ref{proposition:high-prob-deviation}, we have thus shown that
inequality~\eqref{eqn:one-sided-thing-to-prove} holds with
\begin{equation*}
  \epsilon = 
  2 \radphi \sqrt{\frac{d \log \frac{n}{d}}{n}} + \radphi t
  + \frac{\radphi \alpha}{\sqrt{n}}
\end{equation*}
with
probability at least $1 - e^{-nt^2}$, which is the first claim of
Theorem~\ref{theorem:approximate-conditional-coverage}.


Now we turn to the second claim of
Theorem~\ref{theorem:approximate-conditional-coverage}, which
applies when the scores $\scorerv_i$ are distinct with probability
1.
Recall the definition~\eqref{eqn:g-gradient-calcs} of the subgradient terms
$g_i$, and define the index sets $\mc{I}_+ = \{i \mid \<\what{\theta},
\phi(X_i)\> > \scorerv_i\}$, $\mc{I}_- = \{i \mid \<\what{\theta},
\phi(X_i)\> < \scorerv_i\}$, and $\mc{I}_0 = \{i \mid \<\what{\theta},
\phi(X_i)\> = \scorerv_i\}$.
Then
\begin{align*}
  \lefteqn{\sum_{i = 1}^n \phi(X_i)
    \left(\indic{\scorerv_i > \<\what{\theta}, \phi(X_i)\>}
    - \alpha\right)} \\
  & =
  \sum_{i \in \mc{I}_+ \cup \mc{I}_-}
  \phi(X_i)
  \left(\indic{\scorerv_i > \<\what{\theta}, \phi(X_i)\>}
  - \alpha\right)
  + \sum_{i \in \mc{I}_0}
  \phi(X_i) \left(\indic{\scorerv_i > \<\what{\theta}, \phi(X_i)\>}
  - \alpha \right) \\
  & =
  - \sum_{i \in \mc{I}_+ \cup \mc{I}_-}
  \phi(X_i) g_i
  - \sum_{i \in \mc{I}_0} \phi(X_i) g_i
  - \sum_{i \in \mc{I}_0} \phi(X_i) \left(\alpha - g_i\right)
  = \sum_{i \in \mc{I}_0}
  \left(g_i - \alpha\right) \phi(X_i),
\end{align*}
where we used that $\sum_{i = 1}^n g_i \phi(X_i) = 0$ by construction.
Now, we leverage our assumption that $\scorerv_i$ are distinct with
probability 1.
We see immediately that $\card(\mc{I}_0) \le d$, because with
distinct values $\scorerv_i$ we may satisfy (at most) $d$ linear equalities,
and so
\begin{align*}
  \ltwobigg{\frac{1}{n} \sum_{i = 1}^n
    \phi(X_i) \left(
    \indic{\scorerv_i > \<\what{\theta}, \phi(X_i)\>} - \alpha
    \right)}
  \le \frac{\card(\mc{I}_0)}{n}
  \radphi
  \le \frac{d}{n} \radphi.
\end{align*}

Because Proposition~\ref{proposition:high-prob-deviation}
controls the fluctuations of the process $\theta
\mapsto \phi(x) \indic{\scoreval > \<\phi(x), \theta\>}$,
we obtain that with probability at least
$1 - e^{-nt^2}$,
\begin{align*}
  \lefteqn{\ltwobigg{\E\left[\phi(X)
        \indic{\scorerv > \<\what{\theta}, \phi(X)\>}\right]
      - \alpha \E[\phi(X)]}} \\
  & \le
  \ltwobigg{\frac{1}{n} \sum_{i = 1}^n
    \phi(X_i) \left(
    \indic{\scorerv_i > \<\what{\theta}, \phi(X_i)\>} - \alpha
    \right)}
  + 2 \radphi \sqrt{\frac{d}{n} \log \frac{n}{d}}
  + \frac{\radphi \alpha}{\sqrt{n}} + \radphi t \\
  & \le \radphi \frac{d}{n} + 3 \radphi \sqrt{\frac{d}{n} \log \frac{n}{d}}
  + \radphi t.
\end{align*}
This completes the proof of the theorem.




\section{Sharper and rate-optimal approximate conditional bounds}

The bounds Theorem~\ref{theorem:approximate-conditional-coverage} provides
do not reflect the sharper coverage possible in the purely marginal case
that Proposition~\ref{proposition:distinct-scores} provides.
By leveraging empirical process variants of the
Bernstein concentration inequalities we use to prove
Proposition~\ref{proposition:distinct-scores},
we can achieve sharper bounds on weighted coverage, where
the sharper bounds depend on the expectations and variances
of the linear functionals $x \mapsto \<u, \phi(x)\>$ themselves.
As a consequence of our results, in terms of achieving approximate
conditional coverage (i.e., weighted coverage as in
Definition~\ref{definition:weighted-coverage}),
the empirical estimator~\eqref{eqn:empirical-quantile-estimator}
is minimax rate optimal; we discuss this
after Theorem~\ref{theorem:two-sided-sharp}.

To state our results, assume that $\ball \subset \R^d$ is an arbitrary but
bounded set of vectors, and define
\begin{equation*}
  \radphi(u) \defeq \sup_{x \in \mc{X}} |\<u, \phi(x)\>|
  ~~ \mbox{and} ~~
  \radphi \defeq \sup_{u \in \ball} \radphi(u).
\end{equation*}
We can then extend Proposition~\ref{proposition:distinct-scores} to weighted
conditional coverage (Def.~\ref{definition:weighted-coverage}), conditional
on the sample.
We defer their proofs, presenting the building blocks
common to both in Section~\ref{sec:proof-sharp-building-blocks},
then specializing in Sections~\ref{sec:proof-one-sided-sharp}
and~\ref{sec:proof-two-sided-sharp}, respectively.

\begin{theorem}
  \label{theorem:one-sided-sharp}
  Let $K_n = 1 + \log_2 n$.
  Then there exists a numerical constant $c < \infty$ such that
  for all $t \ge 0$,
  with probability at least $1 - 2 K_n e^{-t} - e^{-d \log n - t}$,
  \begin{align*}
    \lefteqn{\E\left[\<u, \phi(X_{n+1})\>
      \left(\indics{Y_{n + 1} \not \in \what{C}(X_{n+1})}
      - \alpha\right)
      \mid P_n \right]} \\
    & \le c \left[\sqrt{\radphi(u) \alpha \cdot \E[\<u, \phi(X)\>]}
      \sqrt{\frac{d \log n + t}{n}}
      + \radphi \frac{d \log n + t}{n}\right]
  \end{align*}
  simultaneously for all $u \in \ball$ such that
  $\<u, \phi(x)\> \ge 0$ for all $x$.
  If additionally the scores $\scorerv_i$ are distinct
  with probability 1, then with the same probability,
  \begin{align*}
    \lefteqn{\E\left[\<u, \phi(X_{n+1})\>
      \left(\indics{Y_{n + 1} \not \in \what{C}(X_{n+1})}
      - \alpha\right)
      \mid P_n \right]} \\
    & \ge - c \left[\sqrt{\radphi(u) \alpha \cdot \E[\<u, \phi(X)\>]}
      \sqrt{\frac{d \log n + t}{n}}
      + \radphi \frac{d \log n + t}{n}\right]
  \end{align*}
  simultaneously for all $u \in \ball$ such that
  $\<u, \phi(x)\> \ge 0$ for all $x$.
\end{theorem}
\noindent
Simplifying the statement and ignoring higher-order terms, we can obtain a
guarantee for weighted coverage~\eqref{eqn:cherian-gibbs-guarantee}: for the
class $\mc{W} = \{w(x) = \<u, \phi(x)\>\}_{u \in \R^d}$, with probability $1
- e^{-t}$,
\begin{equation*}
  \P_w(Y_{n + 1} \in \what{C}_{n + 1})
  \ge 1 - \alpha - O(1) \left[\sqrt{\frac{\alpha}{\E[w(X)]}}
    \cdot \sqrt{\frac{d \log n + t}{n}}\right]
\end{equation*}
simultaneously for $w \ge 0$
with the normalization that $w(x) = \<u, \phi(x)\>$ for
a $u$ satisfying $\radphi(u) = 1$.

Applying the theorem to group indicators, meaning
that we have a collection of groups $\mc{G} \subset 2^{\mc{X}}$, and
the feature mapping $\phi(x) = (\indic{x \in G})_{G \in \mc{G}}$,
we have the following corollary.
\begin{corollary}
  Assume that $\phi(x) = (\indic{x \in G})_{G \in \mc{G}}$,
  and let $d = \card(\mc{G})$.
  Then with probability at least
  $1 - 3 \log_2 n e^{-t}$,
  \begin{equation*}
    \P(Y_{n + 1} \not \in \what{C}(X_{n + 1}) \mid X_{n + 1} \in G,
    P_n)
    \le \alpha +
    c \left[\sqrt{\frac{\alpha}{\P(X_{n + 1} \in G)}
        \frac{d \log n + t}{n}} +
      \frac{d \log n + t}{\P(X_{n + 1} \in G) \cdot n} \right]
  \end{equation*}
  simultaneously for all $G \in \mc{G}$.
\end{corollary}
\noindent
The result follows immediately upon considering the standard basis
vectors $u = e_i$.
Comparing this to Corollary~\ref{corollary:group-conditional-baby},
we see a much sharper deviation guarantee.

When the scores $\scorerv = \scoreval(X, Y)$ are distinct with probability
1, we achieve two sided bounds extending
Theorem~\ref{theorem:one-sided-sharp}, as in
Proposition~\ref{proposition:distinct-scores}.
The next theorem provides an exemplar result.

\begin{theorem}
  \label{theorem:two-sided-sharp}
  Let the conditions of Theorem~\ref{theorem:one-sided-sharp} hold,
  except assume that $\scorerv_i$ are distinct with probability 1,
  and that the mapping $\phi(x)$ includes a constant bias
  term $\phi_1(x) = 1$.
  Then there exists a numerical
  constant $c < \infty$ such that for all
  $t \ge 0$, with probability at least
  $1 - 2 K_n e^{-t} - e^{-d \log n -t}$,
  simultaneously for all $u \in \ball$,
  \begin{align*}
    \left|\E\left[\<u, \phi(X)\> \left(\indics{Y_{n + 1}
          \not \in \what{C}(X_{n + 1})} - \alpha\right)
        \mid P_n \right]\right|
    & \le c \left[\radphi(u) \sqrt{\alpha}
      \sqrt{\frac{d \log n + t}{n}}
      + \radphi\frac{d \log n + t}{n} \right].
  \end{align*}
\end{theorem}
\noindent
The conclusion is weaker than that of Theorem~\ref{theorem:one-sided-sharp},
as it replaces $\sqrt{\radphi(u) \E[\<u, \phi(X)\>]}$ with $\radphi(u)$.

These convergence guarantees are sharp to within logarithmic factors,
and appear to capture the correct dependence on $\alpha$ and the
weight functions $\mc{W}$.
Indeed, assume that the estimated
confidence set $\what{C}_n$ takes the form
$\what{C}_n(x) = \{y \mid \scoreval(x, y) \le \what{h}(x)\}$ or
$\what{C}_n(x) = \{y \mid \what{a}(x) \le \scoreval(x, y)
\le \what{b}(x)\}$ for \emph{some} estimated functions
$\what{h}$, $\what{a}$, or $\what{b}$.
\citet{ArecesChDuKu24} show that for \emph{any} class of functions $\mc{W}$
mapping $\mc{X}$ to $\{\pm 1\}$ with VC-dimension $d$, there exists a
sampling distribution $P$ for which $\scorerv \mid X$ has a continuous
density and such that with constant probability over the draw of $P_n$,
\begin{equation*}
  \left|\E\left[w(X_{n + 1}) \indic{Y_{n + 1} \not \in \what{C}_n(X_{n+1})}
    - w(X_{n + 1}) \alpha  \mid P_n \right]\right|
  \ge c \sqrt{\frac{d \alpha(1 - \alpha)}{n}},
\end{equation*}
where $c > 0$ is a universal constant.
To compare this with Theorems~\ref{theorem:one-sided-sharp}
and~\ref{theorem:two-sided-sharp}, let $\{G_1, \ldots, G_d\}$, $G_j \subset
\mc{X}$, be a partition of $\mc{X}$ into $d$ groups, and define the group
feature mapping $\phi(x) = [\indic{x \in G_j}]_{j = 1}^d$.
Then the class of linear functionals $\mc{W} = \{w \mid w(x) = \<u,
\phi(x)\>\}_{u \in \R^d}$ has VC-dimension $d$, as does its restriction
$\mc{W}_1 = \{w \mid w(x) = \<u, \phi(x)\>\}_{u \in \{\pm 1\}^d}$, where $w
\in \mc{W}_1$ satisfies $w(x) \in \{\pm 1\}$.
Theorems~\ref{theorem:one-sided-sharp} and~\ref{theorem:two-sided-sharp},
conversely, demonstrate that for
$\ball_1 = \{u \mid \lone{u} \le 1\}$,
we have
\begin{equation*}
  \left|\E\left[
      w(X_{n + 1})
      \left(\indic{Y_{n + 1} \not \in \what{C}_n(X_{n + 1})}
      - \alpha\right)
      \mid P_n\right]\right|
  \le c \sqrt{\frac{\alpha(1 - \alpha) }{n}}
  \cdot \sqrt{d \log n + t}
\end{equation*}
with probability at least $1 - e^{-t}$ simultaneously for all $w(x) = \<u,
\phi(x)\>$ for some $u \in \ball_1$, as long as $\alpha < \half$ (where we
use $1 - \alpha \ge \half$ and wrap constants together for a cleaner
statement).



\subsection{Proof of Theorems~\ref{theorem:one-sided-sharp}
  and~\ref{theorem:two-sided-sharp}: building blocks}
\label{sec:proof-sharp-building-blocks}

Our proof leverages a combination of Talagrand's concentration inequalities
for empirical processes, a VC-dimension calculation, and localized
Rademacher complexities~\cite{BartlettBoMe05, Koltchinskii06a}.
We begin with the form of Talagrand's empirical
process inequality with constants due to \citet{Bousquet02thesis}.

\begin{lemma}[Talagrand's empirical process inequality]
  \label{lemma:talagrand}
  Let $\mc{F}$ be a countable class of functions with
  $Pf = 0$ and $\linf{f} \le b$ for $f \in \mc{F}$. Let $Z = \sup_{f \in \mc{F}}
  P_n f$ and $\sigma^2 = \sigma^2(\mc{F}) = \sup_{f \in \mc{F}} Pf^2$.
  Define $v^2 \defeq \sigma^2 + 2b \E[Z]$.
  Then for $t \ge 0$,
  \begin{equation*}
    \P\left(Z \ge \E[Z] + \sqrt{2 v^2 t} + b\frac{t}{3}\right)
    \le e^{-nt}.
  \end{equation*}
\end{lemma}

Because we will consider functions of the form $f(x, \scoreval) = \<u,
\phi(x)\> \indic{\<\theta, \phi(x)\> > \scoreval}$, we also require
some control over the complexity of such rank-one-like products.
\begin{lemma}
  \label{lemma:product-vc}
  Let $\mc{H}$ and $\mc{G}$ be classes of functions with
  VC-dimensions $d_1$ and $d_2$, respectively. Then
  the classes of functions
  \begin{equation*}
    \mc{F} \defeq \{f \mid f(x) = g(x) \indic{h(x) > 0}\}
    ~~ \mbox{and} ~~
    \mc{F}_+ \defeq \{f \mid f(x) = g(x) \indic{h(x) > 0} - c g(x)\}
  \end{equation*}
  where $c$ is a constant
  have VC-dimension $O(1)(d_1 + d_2)$.
\end{lemma}
\begin{proof}
  For a set of points $x_1, \ldots, x_n$, let $\mc{S}(x_1^n, \mc{H}) =
  \{\indic{h(x_i) > 0}\}_{i = 1}^n$ be the set of ``sign'' vectors that $h$
  realizes.
  By definition of the VC-dimension and the Sauer-Shelah lemma, this set has
  cardinality at most $\sum_{i = 0}^{d_1} \binom{n}{i} \le
  (\frac{ne}{d_1})^{d_1}$.
  Similarly, the set of signs
  \begin{equation*}
    \mc{S}(x_1^n, \mc{F})
    = \{\sign(g(x_i)) \cdot \indic{h(x_i) > 0}\}_{i = 1}^n
  \end{equation*}
  has cardinality at most $\sum_{i = 0}^{d_1} \binom{n}{i} \cdot \sum_{i =
    0}^{d_2} \binom{n}{i} \le (\frac{ne}{d_1})^{d_1}
  (\frac{ne}{d_2})^{d_2}$.
  If $n$ is large enough that
  \begin{equation*}
    \left(\frac{ne}{d_1}\right)^{d_1} \cdot \left(\frac{ne}{d_2}\right)^{d_2}
    < 2^n,
  \end{equation*}
  then certainly $\mc{F}$ cannot shatter $n$ points; this occurs
  once $n \ge c \cdot (d_1 + d_2)$ for some numerical constant $c$.
  For the second class the argument is similar.
\end{proof}

\newcommand{\rademacher}{\mathfrak{R}}

For the next lemma, our main technical building block for convergence, we
consider the class of functions $\mc{F}$ indexed by $u \in \ball$ and $h \in
\mc{H}$, where $\mc{H}$ is a class with VC-dimension at most $d$,
with
\begin{equation}
  \label{eqn:product-function-class}
  f(x, \scoreval) = f_{u,h}(x, \scoreval)
  \defeq \<u, \phi(x)\> \indic{\scoreval > h(x)}.
\end{equation}
Each of these functions evidently satisfies
$|f(x, \scoreval)| \le \radphi$.
The variance proxy
\begin{equation}
  \label{eqn:variance-proxy}
  v^2(u, h) \defeq v^2(f_{u,h}) = P f_{u, h}(X, \scorerv)^2
  = \E[\<\phi(X), u\>^2 \indic{\scorerv > h(X)}]
\end{equation}
and its empirical variant
\begin{equation*}
  v_n^2(u, h) = P_n f_{u,h}(X, \scorerv)^2
  = \frac{1}{n} \sum_{i = 1}^n \<\phi(X_i), u\>^2 \indic{\scorerv_i > h(X_i)}.
\end{equation*}
will allow us to bound deviations of $P_n f$ from $P f$ relative
to $v^2(f)$.

For later use, we
recall the \emph{empirical Rademacher complexity} of a function class
$\mc{F}$,
\begin{equation*}
  \rademacher_n(\mc{F})
  \defeq \frac{1}{n} \E\left[\sup_{f \in \mc{F}}
    \sum_{i = 1}^n \varepsilon_i f(X_i) \mid X_1^n\right],
\end{equation*}
where $\varepsilon_i \simiid \uniform\{\pm 1\}$ are random signs.
In some cases, we will require \emph{localized Rademacher
complexities}~\cite{BartlettBoMe05, Koltchinskii06a} around
a class $\mc{F}_r \defeq \{f \mid Pf^2 \le r^2\}$, which
contains functions of small variance, allowing us to ``relativize''
bounds.
\citet[Proof of Corollary 3.7]{BartlettBoMe05} demonstrate the following.
\begin{lemma}
  Let $\mc{F}$ be a star-convex collection of functions,
  meaning that $f \in \mc{F}$ implies $\lambda f \in \mc{F}$
  for $\lambda \in [0, 1]$, and assume that $\sup_x |f(x)| \le b$
  and $\mc{F}$ has VC-dimension $d$. Then
  \begin{equation}
    \label{eqn:rademacher-vc-control}
    \E\left[\rademacher_n(\mc{F}_r)\right]
    \le \frac{\radphi}{n} + 
    c r \sqrt{\frac{d}{n} \log\frac{\radphi}{r}}
    ~~ \mbox{if}~~
    r^2 > \radphi^2 \frac{d}{n} \log \frac{n}{d},
  \end{equation}
  where $c < \infty$ is a numerical constant.  
\end{lemma}

We will combine the VC-bound in Lemma~\ref{lemma:product-vc},
the version of
Talagrand's empirical process inequality in Lemma~\ref{lemma:talagrand}, and
a localization argument on Rademacher complexities
via inequality~\eqref{eqn:rademacher-vc-control}
to prove the following lemma in Appendix~\ref{sec:proof-peel-your-potatoes}.
\begin{lemma}
  \label{lemma:peel-your-potatoes}
  Let $\mc{F}$ be the class of functions~\eqref{eqn:product-function-class}.
  Let $K_n = 1 + \log_2 n$. Then
  for all $t \ge 0$, with probability at least
  $1 - K_n e^{-t}$ over the draw of the sample $P_n$,
  \begin{equation*}
    |(P_n - P) f|
    \le c \left[v(f) \sqrt{\frac{d \log n + t}{n}}
    + \radphi \frac{d \log n + t}{n}\right]
  \end{equation*}
  simultaneously for all $f \in \mc{F}$, where $c < \infty$ is a numerical
  constant.
  In addition, with the same probability,
  \begin{equation*}
    \left|(P_n - P)\<u, \phi(X)\>\right|
    \le c \left[\sqrt{P \<u, \phi\>^2}
      \sqrt{\frac{d \log n + t}{n}}
      + \radphi \frac{d \log n + t}{n}
      \right]
  \end{equation*}
  simultaneously for all $u \in \ball$.
\end{lemma}

Next we present a version of a result appearing as \cite[Theorem
  14.12]{Wainwright19} (the result there assumes functions are mean-zero,
but an inspection of the proof shows this is unnecessary); see also the
results of \cite{Mendelson14} and~\cite[Proof of Proposition
  1]{DuchiRu18a}. These show that second moments satisfy one-sided
concentration bounds with high probability as soon as
we have the fourth moment condition
\begin{equation}
  \label{eqn:four-moments}
  \E[f^4(X, \scorerv)] \le b^2 \E[f^2(X, \scorerv)]
  ~~ \mbox{for~all~} f \in \mc{F}.
\end{equation}
For the setting we consider, where $\mc{F}$ consists
of product functions~\eqref{eqn:product-function-class},
inequality~\eqref{eqn:four-moments} immediately
holds with $b = \radphi$, though tighter constants may be possible.
\begin{lemma}
  \label{lemma:no-concentration-lower}
  There exist numerical constants $0 < c$ and $C < \infty$ such that
  the following holds. Let inequality~\eqref{eqn:four-moments} hold and
  for $\mc{F}_r = \{f \mid Pf^2 \le r^2\}$, let
  $r$ satisfy $\E[\rademacher_n(\mc{F}_r)] \le \frac{r^2}{C b}$. Then
  with probability at least $1 - e^{-c n r^2 / b^2}$,
  \begin{equation*}
    P_n f^2 \ge \half P f^2
    ~~ \mbox{simultaneously~for~all~} f
    ~ \mbox{s.t.}~ v(f) \ge r.
  \end{equation*}
\end{lemma}

Inequality~\eqref{eqn:rademacher-vc-control} shows the conclusions of
of Lemma~\ref{lemma:no-concentration-lower} hold if the radius $r$
satisfies
\begin{equation*}
  r \sqrt{\frac{d}{n} \log \frac{n}{d}}
  \lesssim \frac{r^2}{\radphi}
  ~~ \mbox{or} ~~
  r^2 \gtrsim \radphi^2 \cdot \frac{d}{n} \log \frac{n}{d}.
\end{equation*}
We then obtain the following consequence:
\begin{lemma}
  \label{lemma:specialize-no-concentration}
  Let $r^2 \gtrsim \radphi^2 \frac{d}{n} \log \frac{n}{d}$.
  Then with probability at least $1 - e^{-c n r^2 / \radphi^2}$,
  \begin{equation*}
    P_n\<u, \phi(X)\>^2 \indic{\scorerv > h(X)}
    \ge \half P\<u, \phi(X)\>^2 \indic{\scorerv > h(X)}
  \end{equation*}
  simultaneously over $u \in \ball$ and $h$ such that
  $P\<u, \phi(X)\>^2 \indic{\scorerv > h(X)} \ge r^2$.
\end{lemma}

Now, let $\what{h} = \<\what{\theta}, \phi(\cdot)\>$, where $\what{\theta}$
solves the problem~\eqref{eqn:empirical-quantile-estimator}.
Then simultaneously for all $u \in \ball$,
with probability at least $1 - K_n e^{-t}$,
\begin{equation}
  \label{eqn:bart-ride}
  \left|(P_n - P)\<u, \phi(X)\> \indic{\scorerv > \what{h}(X)}\right|
  \le c \left[v(\what{h}, u)
    \sqrt{\frac{d \log n + t}{n}}
    + \radphi \frac{d \log n + t}{n}\right]
\end{equation}
by Lemma~\ref{lemma:peel-your-potatoes}.
Moreover, for $r^2 \gtrsim \radphi^2 \frac{d}{n} \log \frac{n}{d}$,
either $v(\what{h}, u) \le r$ or
\begin{equation*}
  v^2(\what{h}, u)
  \le 2 P_n \<\phi(X), u\>^2 \indic{\scorerv > \what{h}(X)}
\end{equation*}
by Lemma~\ref{lemma:specialize-no-concentration} (with the appropriate
probability $1 - e^{-cr^2 / \radphi^2}$).

\subsection{Proof of Theorem~\ref{theorem:one-sided-sharp}: nonnegative
  weights}
\label{sec:proof-one-sided-sharp}

We now specialize our development to the particular
cases that $\<u, \phi(x)\> \ge 0$ for all $x \in
\mc{X}$.
First, we leverage the particular structure of the quantile
loss to give a non-probabilistic bound on the empirical
weights.
\begin{lemma}
  \label{lemma:one-directional-coverageish}
  Let $u$ be such that $\<u, \phi(x)\> \ge 0$ for all $x \in \mc{X}$.
  Then
  \begin{equation*}
    P_n \<\phi(X), u\> \indic{\scorerv > \what{h}(X)}
    \le \alpha P_n \<\phi(X), u\>.
  \end{equation*}
  If additionally $\scorerv_i$ are all distinct, then
  \begin{equation*}
    P_n \<\phi(X), u\> \indic{\scorerv > \what{h}(X)}
    \ge \alpha P_n \<\phi(X), u\> - \radphi(u) \frac{d}{n}.
  \end{equation*}
\end{lemma}
\begin{proof}
  The directional derivative
  $\loss'_\alpha(t; 1) \defeq \lim_{\delta \downarrow 0}
  \frac{\loss_\alpha(t + \delta) - \loss_\alpha(t)}{\delta}
  = \indic{t \ge 0} - (1 - \alpha)$.
  Then
  \begin{align*}
    P_n \<\phi(X), u\> \indic{\scorerv > \what{h}(X)}
    & = P_n\<\phi(X), u\>
    \left(1 - \alpha
    - \indic{\scorerv \le \what{h}(X)}\right)
    + P_n \<u, \phi(X)\> \alpha \\
    & = 
    P_n \<\phi(X), u\>
    \left(-\loss'_\alpha(\what{h}(X) - \scorerv; 1)\right)
    + \alpha P_n\<u, \phi(X)\>.
  \end{align*}
  Letting $L_n(h) = P_n \loss_\alpha(h(X) - \scorerv)$, we now use that
  directional derivatives are positively homogeneous~\cite{HiriartUrrutyLe93}
  and that by assumption $\what{h}$ minimizes
  $P_n \loss_\alpha(h(X) - \scorerv)$ over
  functions of the form $h(x) = \<\theta, \phi(x)\>$ to obtain
  \begin{equation*}
    P_n \<\phi(X), u\>
    \left(-\loss'_\alpha(\what{h}(X) - \scorerv; 1)\right)
    = - P_n\loss'_\alpha(\what{h}(X) - \scorerv; \<\phi(X), u\>)
    = - L_n'(\what{h}(X); u).
  \end{equation*}
  But of course, as $\what{h}$ minimizes $L_n$,
  we have $L_n'(\what{h}(X); u) \ge 0$ for all $u$, and so
  \begin{equation*}
    P_n\<\phi(X), u\> \indic{\scorerv > \what{h}(X)}
    \le \alpha P_n \<u, \phi(X)\>.
  \end{equation*}

  If $\scorerv_i$ are all distinct, then considering the
  left directional derivative, we also have
  \begin{align*}
    P_n\<\phi(X), u\> \indic{\scorerv \ge \what{h}(X)} \ge
    \alpha P_n \<u, \phi(X)\>.
  \end{align*}
  If $\mc{I}_0 = \{i \mid \what{h}(X_i) = \scorerv_i\}$,
  then $\card(\mc{I}_0) \le d$, and so
  \begin{equation*}
    0 \ge P_n \<\phi(X), u\>
    \left(\indic{\scorerv > \what{h}(X)}
    - \indic{\scorerv \ge \what{h}(X)}\right)
    = -P_n \<\phi(X), u\> \indic{\scorerv = \what{h}(X)}
    \ge -\radphi(u) \frac{d}{n}.
  \end{equation*}
  Rearranging and performing a bit of algebra, we
  obtain the second claim of the lemma.
\end{proof}

From the lemma, we see that
\begin{align}
  \lefteqn{
    P \<u, \phi(X)\> \left(\indic{\scorerv > \what{h}(X)}
    - \alpha\right)} \nonumber \\
  & = (P - P_n)\<u, \phi(X)\> \left(\indic{\scorerv > \what{h}(X)}
  - \alpha\right)
  + P_n \<u, \phi(X)\> \left(\indic{\scorerv > \what{h}(X)}
  - \alpha\right) \nonumber \\
  & \le (P - P_n) \<u, \phi(X)\> \indic{\scorerv > \what{h}(X)}
  - \alpha (P - P_n) \<u, \phi(X)\>
  \label{eqn:incorporate-the-alpha}
\end{align}
by Lemma~\ref{lemma:one-directional-coverageish}.
Additionally, the lemma implies that
\begin{align*}
  P_n\<\phi(X), u\>^2 \indic{\scorerv > \what{h}(X)}
  & \le \radphi(u) P_n \<\phi(X), u\> \indic{\scorerv > \what{h}(X)}
  \le \radphi(u) \alpha P_n \<\phi(X), u\>.
\end{align*}
We use this to control the first term in the
expansion~\eqref{eqn:incorporate-the-alpha}
by combining these bounds with inequality~\eqref{eqn:bart-ride}
and considering
that $v(\what{h}, u) \le r$ or $v(\what{h}, u) > r$
where $r^2 = O(1) \radphi^2 \frac{d}{n} \log \frac{n}{d}$.
In the latter, we have
$v^2(\what{h}, u) \le c \radphi(u) \alpha P_n\<\phi(X), u\>$.
We have therefore shown that for
any $r^2 \gtrsim \frac{d}{n} \log \frac{n}{d}$,
with probability at
least $1 - K_n e^{-t} - e^{-n r^2}$,
for all $u \in \ball$
with $\<u, \phi(x)\> \ge 0$,
\begin{align}
  \label{eqn:ashby-stop}
  \lefteqn{\left|(P_n - P)\<u, \phi(X)\> \indic{\scorerv > \what{h}(X)}
    \right|}
  \\
  & \le c \left[
    \left(\sqrt{\radphi(u) \alpha P_n \<u, \phi(X)\>} + \radphi r\right)
    \sqrt{\frac{d \log n + t}{n}}
    + \radphi \frac{d \log n + t}{n}
    \right].
  \nonumber
\end{align}

Applying Lemma~\ref{lemma:peel-your-potatoes} to the quantity
$P_n\<u, \phi(X)\>$ shows that simultaneously
for all $u \in \ball$,
\begin{align*}
  |(P_n - P)\<u, \phi(X)\>| & \le
  c
  \left[\sqrt{\radphi(u) P\<u, \phi(X)\>}
    \sqrt{\frac{d \log n + t}{n}}
    + \radphi \frac{d \log n + t}{n}\right]
\end{align*}
with probability at least $1 - K_n e^{-t}$.
Substituting this into the
bounds~\eqref{eqn:incorporate-the-alpha} and~\eqref{eqn:ashby-stop}, and
ignoring lower-order terms (because $\alpha \le 1$), we obtain the guarantee
that for all $t \ge 0$ and $r^2 \gtrsim \frac{d}{n} \log \frac{n}{d}$, then
with probability at least $1 - 2 K_n e^{-t} - e^{-n r^2}$, for all $u$ such
that $P\<u, \phi(X)\> \ge \radphi \frac{d \log n + t}{n}$,
\begin{align*}
  \E\left[
    \<u, \phi(X)\> \left(\indic{\scorerv > \what{h}(X)} - \alpha\right)
    \mid P_n\right]
  & \le
  c \left[
    \left(\sqrt{\alpha \radphi(u) P\<u, \phi(X)\>} + \radphi r\right)
    \sqrt{\frac{d \log n + t}{n}}
    + \radphi \frac{d \log n + t}{n}
    \right]
\end{align*}
and for all $u$ such that $P\<u, \phi(X)\> \le \radphi
\frac{d \log n + t}{n}$,
\begin{align*}
  \E\left[
    \<u, \phi(X)\> \left(\indic{\scorerv > \what{h}(X)} - \alpha\right)
    \mid P_n\right]
  \le c \radphi \left[r \sqrt{\frac{d \log n + t}{n}}
    + \frac{d \log n + t}{n}\right].
\end{align*}
Combining the inequalities and replacing
$r^2$ with $\frac{d \log n + t}{n}$
gives the first claim of Theorem~\ref{theorem:one-sided-sharp}.

For the second claim, when the scores $\scorerv_i$ are distinct,
note simply that we may replace
inequality~\eqref{eqn:incorporate-the-alpha}
with
\begin{align*}
  \lefteqn{
    P \<u, \phi(X)\> \left(\indic{\scorerv > \what{h}(X)}
    - \alpha\right)} \nonumber \\
  & = (P - P_n)\<u, \phi(X)\> \left(\indic{\scorerv > \what{h}(X)}
  - \alpha\right)
  + P_n \<u, \phi(X)\> \left(\indic{\scorerv > \what{h}(X)}
  - \alpha\right) \nonumber \\
  & \ge (P - P_n) \<u, \phi(X)\> \indic{\scorerv > \what{h}(X)}
  - \alpha (P - P_n) \<u, \phi(X)\>
  - \radphi(u) \frac{d}{n}.
\end{align*}
The remainder of the argument is, \emph{mutatis mutandis}, identical to the
proof of the first claim of the theorem.

\subsection{Proof of Theorem~\ref{theorem:two-sided-sharp}: distinct scores}
\label{sec:proof-two-sided-sharp}

Because of the distinctness of
$\scorerv_i$ and that we assume $\phi_1(x) = 1$ (that is, we include
the constant offset), the optimality conditions for
the quantile loss imply that
\begin{equation*}
  \frac{d}{n} \ge \sum_{i = 1}^n \indic{\scorerv_i > \what{h}(X_i)}
  - \alpha \ge -\frac{d}{n}.
\end{equation*}
So if $\radphi(u) \defeq \sup_{x \in \mc{X}} |\<u, \phi(x)\>|$, then
\begin{equation*}
  P_n\<\phi(X), u\>^2 \indic{\scorerv > \what{h}(X)}
  \le \radphi^2(u) \left(\alpha  + \frac{d}{n} \right).
\end{equation*}
Applying inequality~\eqref{eqn:bart-ride},
we find that with probability at least $1 - K_n e^{-t} - e^{-c n r^2 /
  \radphi^2}$,
\begin{equation*}
  \left|(P_n - P)\<u, \phi(X)\> \indic{\scorerv > \what{h}(X)}\right|
  \le c \left[\left(\radphi(u)\sqrt{\alpha}
    + r\right)\sqrt{\frac{d \log n + t}{n}}
    + \radphi  \frac{d \log n + t}{n}\right].
\end{equation*}
The deviations $\alpha (P_n - P)\<u, \phi(X)\>$ are of smaller
order than this by Lemma~\ref{lemma:peel-your-potatoes}.

\section{Experimental Results}

Our main purpose thus far has been to re-investigate conditional quantile
estimation procedures, providing theoretical bounds for their performance;
there is already substantial practical experience with these methods.
Nonetheless, we can incorporate a few recent theoretical results to enhance
the practical performance of the proposed conformalization procedures,
allowing some additional performance gains, while simultaneously
exhibiting the need for future work.
We consider mostly the difference between the full conformal approach that
\citet{GibbsChCa23} develop and the split-conformal approaches
that simply minimize the empirical
loss~\eqref{eqn:empirical-quantile-estimator}.
Our theoretical results provide no guidance to lower-order corrections
to the desired level $\alpha$ to guarantee (exact) marginal coverage
rather than approximate sample-conditional coverage, and
so we proceed a bit heuristically here, using theoretical
results to motivate modifications of the level $\alpha$ that
do not change the sample-conditional coverage results we provide, but
which turn out to be empirically effective.

\newcommand{\alphadesired}{\alpha_{\textup{des}}}

To motivate our tweaks, recall the classical (unconditional) conformal
approach to achieve exact finite-sample marginal coverage $\P(Y_{n + 1} \in
\what{C}_n(X_{n + 1}))$, where the confidence confidence set $\what{C}_n(x)
= \{y \mid \scoreval(x, y) \le \what{\tau}_n\}$.
Setting $\what{\tau}_n = \quant_{(1 + 1/n) (1 - \alpha)}(\scorerv_1^n)$, the
slightly enlarged quantile, guarantees $(1 - \alpha)$ coverage; this follows
by letting $\scorerv_{(i,n)}$ be the order statistics of $\scorerv_1^n$ and
$\scorerv_{(i, n + 1)}$ those of $\scorerv_1^{n + 1}$, and noting that the
score $\scorerv_{n + 1} \le \scorerv_{(k,n)}$ if and only if $\scorerv_{n +
  1} \le \scorerv_{(k , n + 1)}$~\cite[Lemma 2]{RomanoPaCa19}, so the
inflation by $\frac{n+1}{n}$ is necessary.
Equivalently, if we wish to achieve coverage $(1 - \alphadesired)$ using the
estimator~\eqref{eqn:empirical-quantile-estimator} with feature mapping
$\phi(x) = 1$ fit at level $\alpha$, then $\alpha$ must solve $(1 - \alpha)
= (1 + \frac{1}{n}) (1 - \alphadesired)$, that is, $\alpha = 1 - (1 +
\frac{1}{n})(1 - \alphadesired) = (1 + \frac{1}{n}) \alphadesired -
\frac{1}{n}$.
That is, quantile regression under-covers.

When $\phi : \mc{X} \to \R^d$, it is then natural to heuristically imagine
that the order statistics ought to ``swap orders'' by at most roughly $d$
items and so we ought to target coverage $\frac{n + d}{n}(1 -
\alphadesired)$; unfortunately, it escapes our ability to prove such a
result currently.
Nonetheless, we consider a ``naive'' adaptation of the confidence level,
setting $\alpha$ to solve
\begin{equation}
  \label{eqn:naive-alpha}
  (1 - \alpha) = \left(1 + \frac{d}{n}\right) (1 - \alphadesired),
  ~~ \mbox{or} ~~
  \alpha = \left(1 + \frac{d}{n}\right) \alphadesired - \frac{d}{n},
\end{equation}
and then choosing $\what{\theta}$ to
minimize~\eqref{eqn:empirical-quantile-estimator} with this $\alpha$,
which we term the ``naive'' choice.
\citet{BaiMeWaXi21} give an alternative perspective, where they show
that the actual marginal coverage achieved by quantile regression
at level $\alpha$ in the high-dimensional scaling $d, n \to \infty$ with
$d/n \to \kappa \in (0, 1)$ is
\begin{equation*}
  (1 - \alpha) - \frac{d}{n} \left(\half - \alpha\right) + o(d/n)
\end{equation*}
for $\alpha < \half$, at least when the covariates are Gaussian.
Solving this and ignoring the higher-order term, we recognize that to
achieve desired coverage $\alpha$, we ought (according to this heuristic) to
compute the estimator~\eqref{eqn:empirical-quantile-estimator} using
$\alpha$ solving
\begin{equation}
  \label{eqn:bai-alpha}
  (1 - \alphadesired) = (1 - \alpha) - \frac{d}{n}
  \left(\half - \alpha\right)
  ~~ \mbox{or} ~~
  \alpha = \frac{\alphadesired - \frac{d}{2n}}{1 - \frac{d}{n}}.
\end{equation}
We call the choice~\eqref{eqn:bai-alpha} the ``scaling'' choice.
Neither of the rescalings~\eqref{eqn:naive-alpha} or~\eqref{eqn:bai-alpha}
have any effect on the convergence guarantees our theory provides, as they
are of lower order.

\subsection{Synthetic datasets}
\label{sec:synthetic}

We perform two synthetic experiments that give a
sense of the coverage properties of the methods we have analyzed.
These exploratory experiments help provide justification for the heuristic
corrections to the desired level $\alpha$ we set in the real data
experiments.

\subsubsection{Level rescaling on a simple synthetic dataset}


\newcommand{\noise}{\varepsilon}

For our first experiment, we consider the simple setting of
a standard Gaussian linear model, where we observe
\begin{equation*}
  y_i = \<w\opt, x_i\> + \noise_i,
  ~~ \noise_i \simiid \normal(0, 1)
  ~~ \mbox{and} ~~ x_i \simiid \normal(0, I_d).
\end{equation*}
We mimic the experiment \citet[Fig.~3]{GibbsChCa23} provide,
but we investigate the coverage properties of the
coverage set from the estimator~\eqref{eqn:empirical-quantile-estimator}
with uncorrected $\alpha$ and level $\alpha$ corrected
either naively~\eqref{eqn:naive-alpha} or via the scaling
correction~\eqref{eqn:bai-alpha}.
In all cases, we use the feature map $\phi(x) = (1, \indic{x_1 > 0}, \ldots,
\indic{x_d > 0}) \in \{0, 1\}^{d + 1}$ indicating nonnegative coordinates.
Figure~\ref{fig:simulation-correction-coverage} displays the results of this
experiment for 1000 trials, where in each trial, we draw $w\opt \sim
\uniform(\sphere^{d-1})$, fit a regression estimator $\what{w}$ on a
training dataset of size $n_{\textup{train}} = 100$ using least squares,
then conformalize this predictor using a validation set of size $n$ and
evaluate its coverage on a test dataset of size $n_{\textup{test}} = 10
n_{\textup{train}} = 1000$.
We vary the
ratio $n / d$ of the validation dataset, keeping $d = 20$ fixed.
From the figure, it is clear that the uncorrected confidence
set using $\alpha = \alphadesired = .1$
undercovers, especially when the ratio $n/d < 20$ or so.
The naive correction~\eqref{eqn:naive-alpha} appears to be a bit
conservative, while the scaling correction~\eqref{eqn:bai-alpha}
is more effective.

\begin{figure}[ht]
  \begin{center}
    \begin{tabular}{cc}
      \hspace{-1cm}
      \includegraphics[width=.55\columnwidth]{%
        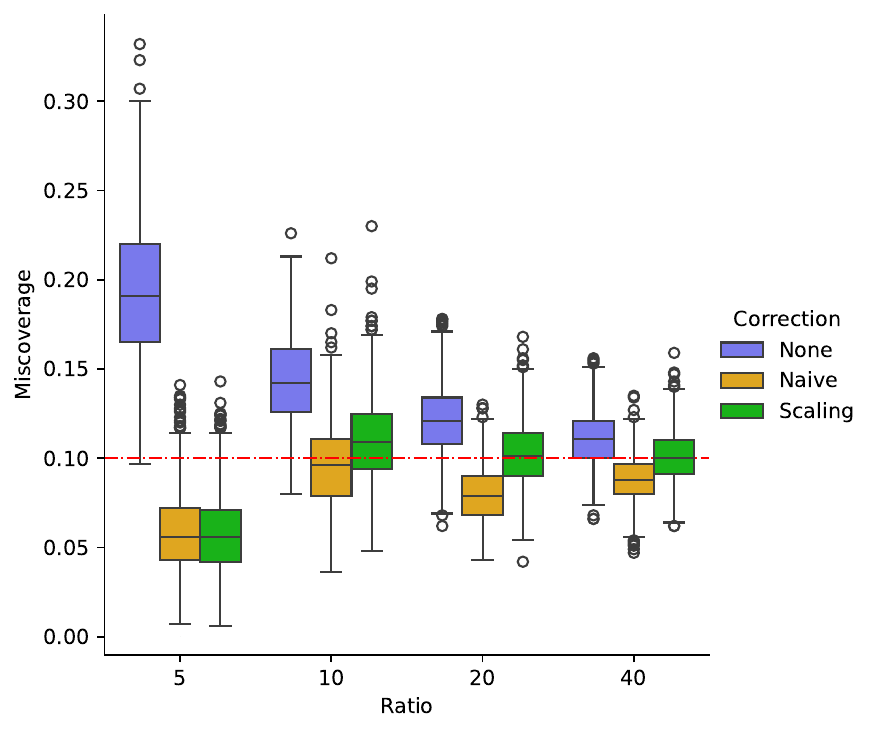} &
      \hspace{-.5cm}
      \includegraphics[width=.55\columnwidth]{%
        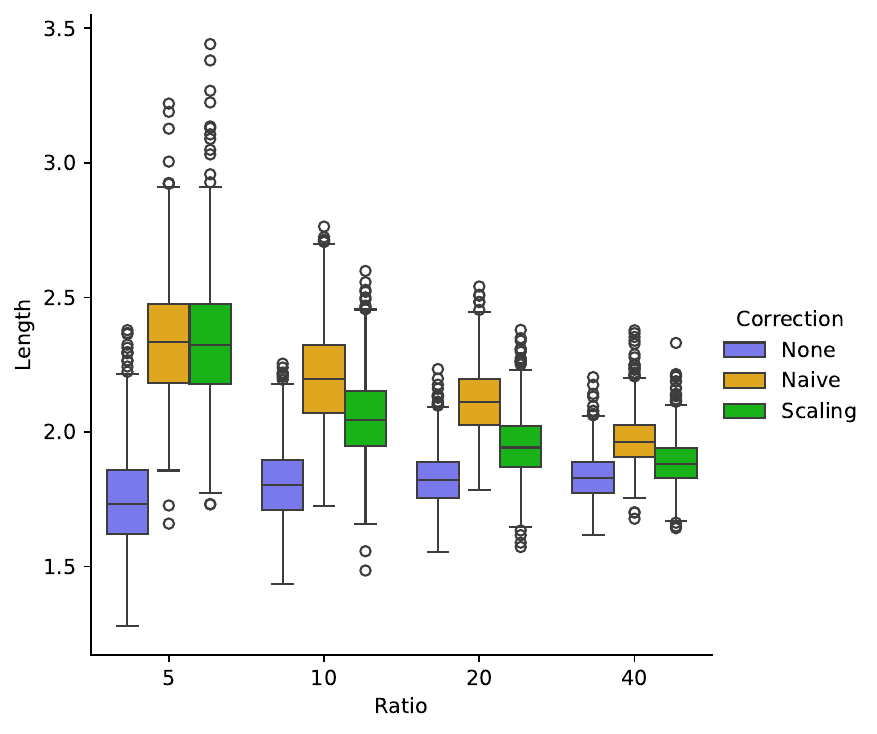} \\
      (a) & (b)
    \end{tabular}
    \caption{\label{fig:simulation-correction-coverage} Impact of the
      correction to $\alpha$ used in fitting the conformal
      predictor~\eqref{eqn:empirical-quantile-estimator} for a desired level
      $\alphadesired = .1$, i.e., 90\% coverage. The ``None'' correction
      uses $\alpha = \alphadesired$, ``Naive'' uses the
      correction~\eqref{eqn:naive-alpha}, and ``Scaling'' uses the
      correction~\eqref{eqn:bai-alpha}.  (a) Coverage rates with the desired
      coverage marked as the red line. (b) Width of predictive intervals
      $\what{C}(x) = \{y \in \R \mid |\what{f}(x) - y| \le
      \what{\theta}^\top \phi(x)\}$.}
  \end{center}
\end{figure}

\subsubsection{Full conformal versus split-conformal predictions}
\label{sec:offline-sin-simulation}

We briefly look at the coverage properties of the full conformalization
method~\eqref{eqn:implicit-full-confidence-set} from the
paper~\citep{GibbsChCa23}, comparing with split-conformal
methods~\eqref{eqn:empirical-quantile-estimator}, on a synthetic regression
dataset we design to have asymmetric mean-zero heteroskedastic noise.
We generate pairs $(X_i, Y_i) \in \R^2$ according to
$Y = f(x) + \noise(x)$,
discretizing $x \in [0, 1]$ into bins
$B_i = \{x \mid \frac{i}{k} \le x < \frac{i + 1}{k}\}$,
$i = 0, \ldots, k - 1$, for $k = 5$.
Within each experiment, we
draw $U_0, U_1 \simiid \uniform[-1, 1]$ and
$\phi_0, \phi_1 \simiid \uniform[\frac{\pi}{4}, 4 \pi]$
to define
\begin{equation*}
  f(x) = U_0 \cos(\phi_0 \cdot x)
  + U_1 \sin(\phi_1 \cdot x).
\end{equation*}
Within the $i$th region $\frac{i}{k} \le x < \frac{i + 1}{k}$ we
set $\lambda_{0,i} = \exp(3 - \frac{3}{k} i)$ and
$\lambda_{1,i} = \exp(4 - \frac{3}{k} i)$, i.e.,
evenly spaced in $\{e^{3}, \ldots, e^0\}$ and $\{e^{4}, \ldots, e^{1}\}$,
and draw
\begin{equation*}
  \noise(x) \sim \begin{cases} \expdist(\lambda_{0,i})
    & \mbox{with probability}~
    \frac{\lambda_{0,i}}{\lambda_{0,i} + \lambda_{1,i}}
    = \frac{1}{1 + e} \\
    -\expdist(\lambda_{1,i})
    & \mbox{otherwise},
  \end{cases}
\end{equation*}
so that $\E[\noise(x) \mid x] = 0$ but the noise is skewed upward,
with variance increasing in $i$.

Figure~\ref{fig:little-silly-experiment} shows the results of this
experiment over 200 independent trials, where in each
experiment we draw a new mean function $f$ and fit it using
a degree 5 polynomial regression on a training
set of size $n_{\textup{train}} = 200$.
The conformalization methods use a group-indicator featurization $\phi(x) =
(1, \indic{x \in B_1}, \ldots, \indic{x \in B_k})$ and confidence sets
$C(x) = \{y \mid \theta_0^\top \phi(x) \le y \le \theta_1^\top
\phi(x)\}$.
Within each trial, we compute miscoverage
proportions $\P(Y \not \in \what{C}(X) \mid X \in B_i)$ for each bin $i$ on
a test set of size 500, drawing a new function $f$.
We vary the size of the validation data $n_{\textup{val}}
= \{10 k, 20k, 40k, 80k, 160k\}$, and
use the scaling correction~\eqref{eqn:bai-alpha} to set $\alpha$ for
the split-conformal method.
The figure plots results for validation sizes $40k$ and $160k$; from the
figure---which is consistent with our other sample sizes and
experiments---we see that when the validation size is large relative to the
dimension of the mapping $\phi$, both methods are similar; for smaller
ratios, the offline method undercovers slightly within the groups, though
its marginal coverage remains near perfect in spite of the very non-Gaussian
data.

\begin{figure}
  \begin{center}
    \begin{tabular}{cc}
      \hspace{-.5cm}
      \includegraphics[width=.5\columnwidth]{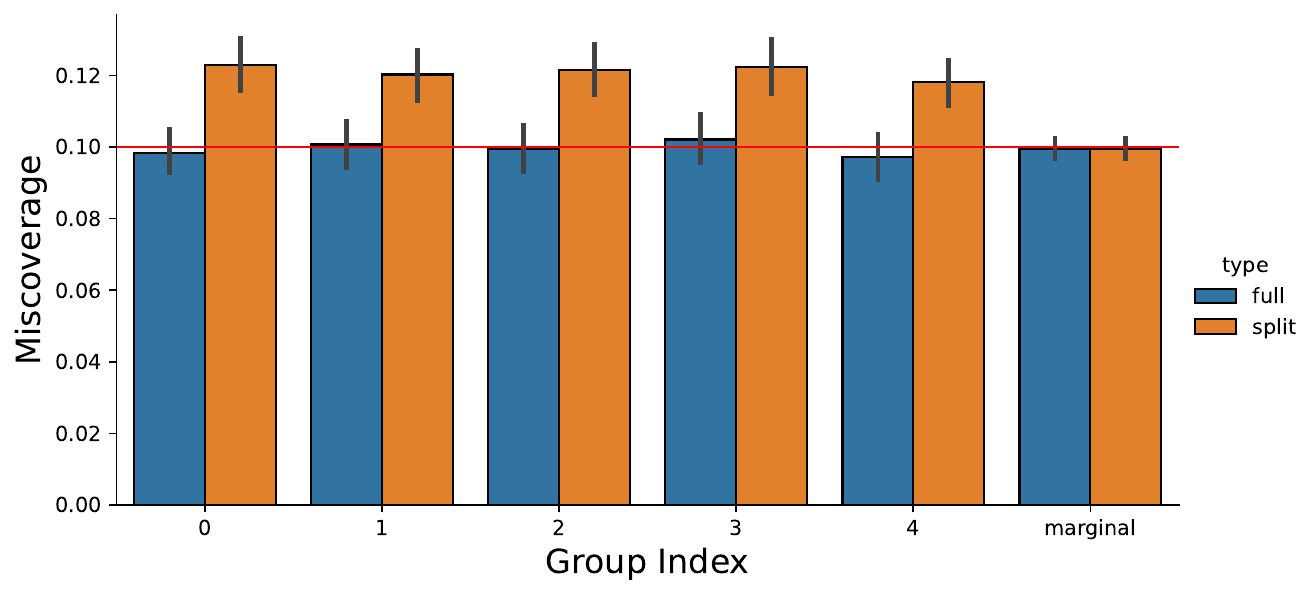}
      &
      \hspace{-.25cm}
      \includegraphics[width=.5\columnwidth]{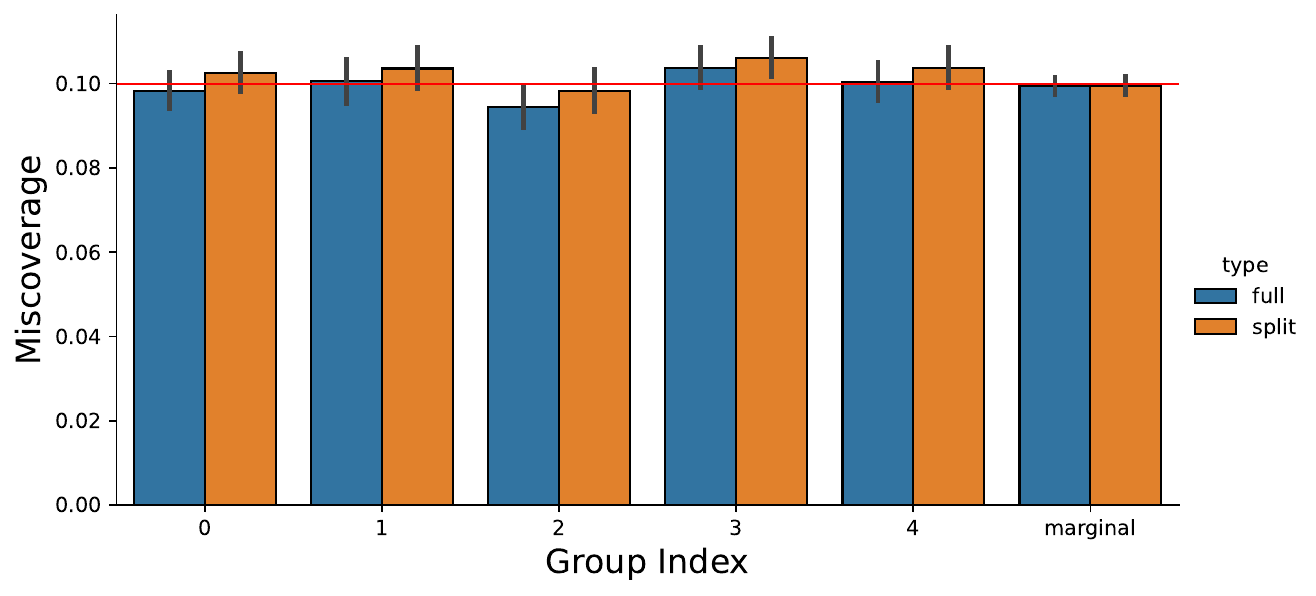} \\
      (a) & (b) \\
      \hspace{-.5cm}
      \includegraphics[width=.5\columnwidth]{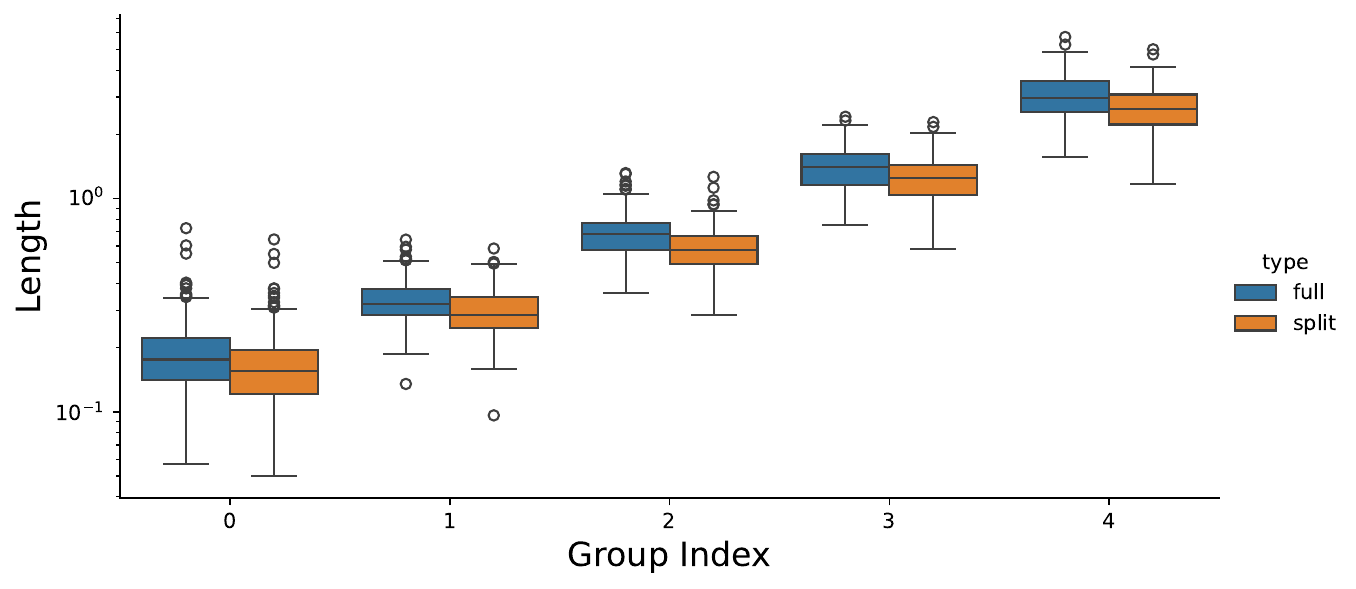} &
      \hspace{-.25cm}
      \includegraphics[width=.5\columnwidth]{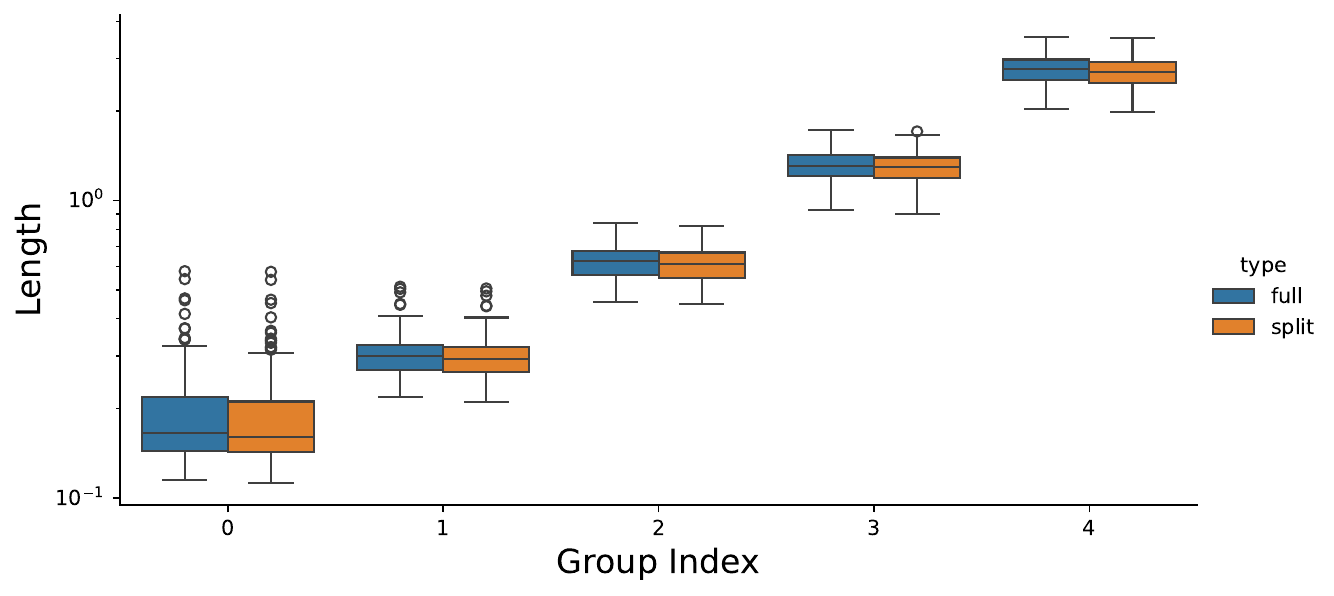} \\
      (c) & (d)
    \end{tabular}
    \caption{
      \label{fig:little-silly-experiment}
      Comparison of full- and split-conformal
      methods on the simulated sinusoidal data of
      Sec.~\ref{sec:offline-sin-simulation} with
      $n_{\textup{train}} = 200$ training examples and target
      miscoverage $\alpha = .1$.
      Plots (a) and (c) use validation
      sample sizes $n_{\textup{val}} = 20 k = 100$,
      while (b) and (d) use
      $n_{\textup{val}} = 160 k = 800$. Plots (a) and
      (b) show miscoverage $\P(Y \not \in \what{C}(X) \mid X \in B_i)$ by
      group $B_i$; plots (c) and (d) prediction interval lengths.
    }
  \end{center}
\end{figure}

We remark in passing that the full conformal method requires roughly $10
\times$ the amount of time to compute predictions as the split conformal
method requires to both fit a quantile prediction model and make its
predictions.
Once the split-conformal quantile model is available---it has been
fit---this difference becomes roughly a factor
of 2000--4000 in our experiments.
For some applications, this may be immaterial; for others, it may be
a substantial expense, suggesting that a decision between
the offline method and the online procedure may boil down to
one of computational feasibility.

\subsection{Prediction on CIFAR-100}

We also perform an exploratory experiment on the CIFAR-100 dataset, a
100-class image classification dataset consisting of 60,000 training
examples and a 10,000 example test set.
We use the output features of a 50 layer ResNet, pre-trained
on the ImageNet dataset~\cite{HeZhReSu16, HeZhReSu16b},
as a $d = 2048$-dimensional input into a multiclass
logistic regression classifier.
We repeat the following experiment 10 times:
\begin{enumerate}[1.]
\item Uniformly randomly split the training examples into a validation set
  of size 10,000 and a model training set of size 50,000, on which we fit a
  linear classifier $\scorefunc : \R^d \to \R^k$, where $\scorefunc_y(x) =
  \<\beta_y, x\>$ is the score assigned to class $y$, using multinomial
  logistic regression.
\item \label{item:project-down}
  Draw a random matrix $W \in \R^{d \times d_0}$, where
  $d_0 = 10$ and $W_{ij} \simiid \normal(0, 1)$,
  and use the validation data with score function
  $\scorefunc(x, y) = \scorefunc_y(x)$ and the lower-dimensional
  mapping $\phi(x) = W^\top x$
  to predict quantiles via $\what{h}(x) = \<\what{\theta}, \phi(x)\>$.
\item Compare the coverage of the full-conformal method,
  standard split conformal with a static threshold,
  meaning confidence sets of the form
  $\what{C}(x) = \{y \in [k] \mid \scorefunc(x, y) \le \what{\tau}\}$,
  and split conformal with the threshold function
  $\what{h}(x)$ fit on the validation data.
  To perform the comparison, we draw subsamples from the test
  $Z_{\textup{test}} = \{(x_i, y_i)\}_{i = 1}^{n_{\textup{test}}}$
  by defining groups $G$ of the form
  \begin{align*}
    G_{j,>} & = \left\{(x, y) \in Z_{\textup{test}}
    \mid  \<w_j, x\> \ge \quant_{.8}(\{\<w_j, x_i\>\}_{i = 1}^{n_{\textup{test}}})
    \right\} ~~ \mbox{and} \\
    G_{j,<} & = \left\{(x, y) \in Z_{\textup{test}}
    \mid  \<w_j, x\> \le \quant_{.2}(\{\<w_j, x_i\>\}_{i = 1}^{n_{\textup{test}}})
    \right\}
  \end{align*}
  for each row $w_j^\top$ of the random dimension reduction matrix
  $W$ from step~\ref{item:project-down}.
\end{enumerate}

Figure~\ref{fig:random-directions} displays the results of this experiment.
In the figure, we notice three main results: first, the static thresholded
sets $\what{C}(x) = \{y \mid \scoreval(x, y) \le \what{\tau}\}$ have
substantially more variability in coverage on the random slices of the
dataset.
Second, the split conformal method and full conformal methods have
similar coverage on each of the slices, with some
slices exhibiting more variability of the full conformal methodology
and some with the split conformal methodology, but all
around the nominal (desired) 90\% coverage level.
Finally, the split conformal methods slightly undercover marginally,
while the full conformal method slighlty overcovers marginally.

\begin{figure}
  \begin{center}
    \includegraphics[width=.8\columnwidth]{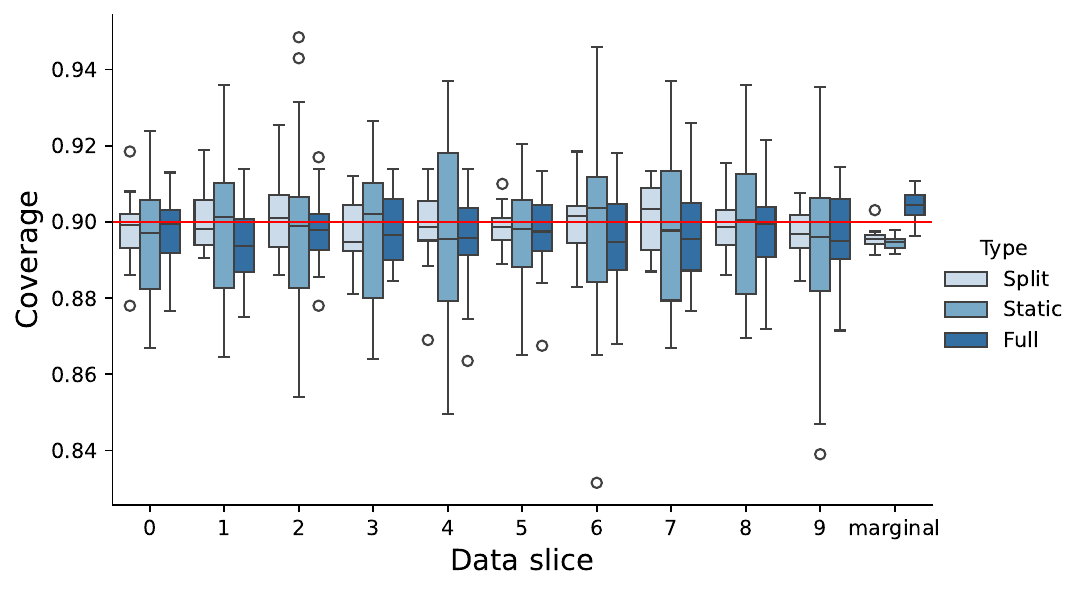}
    \caption{\label{fig:random-directions} Coverage of full conformal, split
      conformal, and static split conformal methods on random 20\%
      ``slices'' of CIFAR-100 data.}
  \end{center}
\end{figure}

\section{Conclusion and discussion}

The results in our experiments, though they are relatively small scale,
appear to be consistent with other experiments we do not present for
brevity.
In brief: split conformal methods with confidence sets using adaptive
thresholds of the form $\what{C}(x) = \{y \mid \scoreval(x, y) \le
\what{h}(x)\}$ can indeed provide stronger coverage than non-adaptive
thresholds.
Moreover, they are \emph{much} faster to compute with than full conformal
methods---in the experiment in Figure~\ref{fig:random-directions},
the split conformal method was roughly 8000$\times$ faster than
the full conformal method.
Additionally, they enjoy strong sample-conditional stability
as well as minimax optimality.

In spite of this, when the adaptive threshold
$\what{h}(x)$ comes from a class of functions that is high-dimensional
relative to the size of the data available for calibration,
these methods can undercover, as they exhibit downard bias in their
coverage.
This bias is easy to correct for a static threshold
$\what{C}(x) = \{y \mid \scoreval(x, y) \le \what{\tau}\}$ by
simply using a slightly larger quantile, however, it is
unclear how to address it in adaptive scenarios.
This makes obtaining a data-adaptive way to compute the coverage bias,
either marginally or along various splits of the validation data,
substantially interesting and a natural direction for future work.
Identifying such an offline correction without relying, as our heuristic
development in Section~\ref{sec:synthetic}, could make these procedures
substantially more practical, by both enjoying the test-time speed of split
conformal methods and the coverage accuracy of full-conformal procedures.

\appendix

\section{Technical proofs}

\subsection{Proof of Lemma~\ref{lemma:bound-deviation-ltwo}}
\label{sec:proof-bound-deviation-ltwo}

When $\ball_2$ is the $\ell_2$-ball,
\begin{equation*}
  \E\left[\sup_{h \in \mc{H}, u \in \ball_2}
    \<u, Z_n(h) - \E[Z_n(h)]\>\right]
  = \E\left[\sup_{h \in \mc{H}} \ltwo{Z_n(h) - \E[Z_n(h)]}
    \right].
\end{equation*}
Performing a typical symmetrization argument, we let $P_n^0 = \frac{1}{n}
\sum_{i = 1}^n \varepsilon_i \pointmass_{X_i, \scorerv_i}$ be the
symmetrized empirical measure, where $\varepsilon_i \simiid \uniform\{\pm
1\}$ are i.i.d.\ Rademacher variables, and define the symmetrized process
$Z_n^0(h) = \frac{1}{n} \sum_{i = 1}^n \varepsilon_i \phi(X_i)
\indic{\scorerv_i > h(X_i)}$.
Then for the (random) set of vectors
$\mc{V}_n = \{(\indic{\scorerv_1 > h(X_1)}, \ldots, \indic{\scorerv_n >
  h(X_n)})\}_{h \in \mc{H}} \subset \{0, 1\}^n$
\begin{align*}
  \E\left[\sup_{h \in \mc{H}} \ltwo{Z_n(h) - \E[Z_n(h)]} \right]
  & \le 2 \E\left[\sup_{h \in \mc{H}} \ltwo{Z_n^0(h)}\right]
  \le 2 \E\left[\max_{v \in \mc{V}_n}
    \ltwobigg{\frac{1}{n} \sum_{i = 1}^n \varepsilon_i \phi(X_i) v_i}\right].
\end{align*}
Now, we recognize that because the vectors $\phi$ lie in a Hilbert
space,
we enjoy certain dimension free concentration guarantees.
In particular, we have for any fixed $v \in \{0, 1\}^n$ that
\begin{equation*}
  \P\left(\ltwobigg{\sum_{i = 1}^n \varepsilon_i \phi(X_i) v_i} \ge t
  \mid X_1^n \right)
  \le 2 \exp\left(-\frac{t^2}{2 \Phi_n^2}\right),
\end{equation*}
where $\Phi_n^2 \defeq \sum_{i = 1}^n \ltwo{\phi(X_i)}^2$ by
\citet[Theorem 3.5]{Pinelis94} (see also~\cite[Corollary
  10]{HowardRaMcSe20}).
In particular, using that for $U$ a nonnegative random variable
$\E[U] = \int_0^\infty \P(U \ge u) du$,
we obtain
\begin{align*}
  \E\left[\max_{v \in \mc{V}_n} \ltwobigg{\sum_{i = 1}^n
      \varepsilon_i \phi(X_i) v_i}
    \mid X_1^n \right]
  & \le \int_0^\infty \P
  \left(\max_{v \in \mc{V}_n}
  \ltwobigg{\sum_{i = 1}^n
    \varepsilon_i \phi(X_i) v_i} \ge t \mid X_1^n \right) dt \\
  & \le t_0
  + 2 \card(\mc{V}_n) \int_{t_0}^\infty
  \exp\left(-\frac{t^2}{2 \Phi_n^2}\right) dt.
\end{align*}
Recognizing the Gaussian tail bound
that
\begin{equation*}
  \int_c^\infty e^{-\frac{t^2}{2 \sigma^2}} dt
  = \sqrt{2 \pi \sigma^2}
  \int_{c/\sigma}^\infty \frac{1}{\sqrt{2\pi}} e^{-z^2 / 2} dz
  \le \sqrt{2 \pi \sigma^2}
  \min\left\{\frac{1}{\sqrt{2\pi}}
  \frac{\sigma}{c}, 1 \right\}
  \exp\left(-\frac{c^2}{2\sigma^2}\right)
\end{equation*}
by Mills' ratio, we see that for any $t_0 \ge 0$,
\begin{align*}
  \E\left[\max_{v \in \mc{V}_n} \ltwobigg{\sum_{i = 1}^n
      \varepsilon_i \phi(X_i) v_i}
    \mid X_1^n \right]
  & \le t_0 + 2 \card(\mc{V}_n)
  \cdot \frac{\Phi_n^2}{t_0}
  \exp\left(-\frac{t_0^2}{2 \Phi_n^2}\right).
\end{align*}

Finally, recognize that $\mc{V}_n$ has cardinality at most $(\frac{e
  n}{k})^{k}$ by the Sauer-Shelah lemma because $\mc{H}$ has VC-dimension
$k$.
Consequently, we may take $t_0^2 = 2 \log \card(\mc{V}_n) \Phi_n^2$ to
obtain the bound
\begin{align*}
  \E\left[\max_{v \in \mc{V}_n} \ltwobigg{\sum_{i = 1}^n
      \varepsilon_i \phi(X_i) v_i}
    \mid X_1^n \right]
  & \le \sqrt{2 \log \card(\mc{V}_n)} \Phi_n
  + \frac{\Phi_n}{\sqrt{2 \log \card(\mc{V}_n)}}
  \le 2 \sqrt{k \log \frac{n e}{k}} \cdot \Phi_n.
\end{align*}
Take expectations over $X_1^n$.

\subsection{Proof of Lemma~\ref{lemma:peel-your-potatoes}}
\label{sec:proof-peel-your-potatoes}

For $r \ge 0$, define the localized class
\begin{equation*}
  \mc{F}_r \defeq
  \left\{f \in \mc{F}
  \mid Pf^2 = \E[\<v, \phi(X)\>^2 \indic{\scorerv
      > h(X)}] \le r^2 \right\}.
\end{equation*}
Note that $\mc{F}_r$ always includes the 0 function and is star-convex,
because if $f \in \mc{F}_r$, then $\lambda f \in \mc{F}_r$ for $\lambda \in
[0, 1]$.
Recalling inequality~\eqref{eqn:rademacher-vc-control},
the second term dominates the first, and so
\begin{equation*}
  \E\left[\rademacher_n(\mc{F}_r)\right]
  \le c r \sqrt{\frac{d}{n} \log \frac{n}{d}}
  ~~ \mbox{if}~~
  r^2 \ge \radphi^2 \frac{d}{n} \log \frac{n}{d}.
\end{equation*}

Define the random variable $Z_n(r) \defeq
\sup_{f \in \mc{F}_r} (P_n - P) f = \sup_{f \in \mc{F}_r} |(P_n - P) f|$, the
equality following by symmetry of $\mc{F}_r$. Then
Talagrand's concentration inequality (Lemma~\ref{lemma:talagrand}) implies
that
\begin{equation*}
  \P\left(Z_n(r) \ge \E[Z_n(r)]
  + \sqrt{2 (r^2 + 2 \radphi \E[Z_n(r)]) t}
  + \frac{\radphi}{3} t \right) \le e^{-nt}
\end{equation*}
for all $t \ge 0$.  Applying a standard symmetrization argument and
inequality~\eqref{eqn:rademacher-vc-control},
we thus obtain that for $r^2 \ge \radphi^2 \frac{d}{n} \log \frac{n}{d}$,
with probability at least $1 - e^{-t}$,
\begin{equation*}
  Z_n(r) \le c r \sqrt{\frac{d \log n}{n}}
  + c \sqrt{\frac{r^2}{n^2} + \frac{\radphi^2}{n}
    r \sqrt{\frac{d \log n}{n}}} \sqrt{t}
  + \frac{\radphi t}{3n}.
\end{equation*}
As the last step, we apply a peeling
argument~\cite{Wainwright19, vandeGeer00}: consider the intervals
\begin{equation*}
  E_k \defeq \openleft{2^{k-1}\frac{\radphi^2 d \log n}{n}}{
    2^k \frac{\radphi^2 d \log n}{n}}
  ~~
  k = 1, 2, \ldots, K_n \defeq \ceil{\log_2 \frac{n}{d \log n}}.
\end{equation*}
Let $\mc{F}^k = \{f \in \mc{F} \mid  Pf^2 \in E_k\}$,
where $\mc{F}^0 = \{f \in \mc{F} \mid Pf^2 \le \frac{d \log n}{n}\}$. Then
evidently $\cup_{k = 0}^{K_n} \mc{F}^k = \mc{F}$, and letting
$r_k^2 = 2^k \radphi^2 \frac{d \log n}{n}$, we have
$\mc{F}_{r_k} \subset \mc{F}^k$. So by a union bound,
with probability at least $1 - (K_n + 1)e^{-t}$,
\begin{subequations}
  \label{eqn:peeling-inequalities}
  \begin{equation}
    Z_n(r_k) \le c r_k \sqrt{\frac{d \log n}{n}} + c
    \sqrt{\frac{r_k^2}{n^2} + \frac{\radphi}{n}
      \sqrt{\frac{r_k^2 d \log n}{n}}} \sqrt{t} + \frac{\radphi}{3n} t
    ~~ \mbox{for~} k = 1, \ldots, K_n
  \end{equation}
  and
  \begin{equation}
    Z_n(r_0) \le c \frac{d \log n}{n}
    + \sqrt{\frac{r_0^2}{n^2} + \frac{\radphi}{n}
      \frac{d \log n}{n}} \sqrt{t} + \frac{\radphi t}{3n}.
  \end{equation}
\end{subequations}

Recall the definition $v^2(f) \defeq Pf^2 = \var(f) + (Pf)^2$. Then
$f \in \mc{F}^k$ implies $\half r_k \le v(f) \le r_k$, so that
on the event that all the inequalities~\eqref{eqn:peeling-inequalities}
hold, then simultaneously for all $f$ satisfying
$v^2(f) \ge \frac{d \log n}{n}$, then
\begin{equation*}
  \left|(P_n - P) f\right|
  \le c v(f) \sqrt{\frac{d \log n}{n}}
  + c \sqrt{\frac{v^2(f)}{n^2}
    + \frac{\radphi}{n} \sqrt{v^2(f) \frac{d \log n}{n}}}
  \sqrt{t} + \frac{\radphi}{3n} t,
\end{equation*}
while for all $f$ with $v^2(f) \le \frac{d \log n}{n}$ we have
\begin{equation*}
  |(P_n - P) f|
  \le c \frac{d \log n}{n}
  + c \sqrt{\frac{d \log n}{n^3}
    + \frac{\radphi}{n} \frac{d \log n}{n}} \sqrt{t} + \frac{\radphi}{3n} t.
\end{equation*}
(To obtain the absolute bounds, we have used that $f \in \mc{F}$ implies
$-f \in \mc{F}$ and each set $\mc{F}_r$ and $\mc{F}^k$ is symmetric.)
Finally, we note that
\begin{align*}
  \sqrt{\frac{v^2(f)}{n^2}
    + \frac{\radphi}{n}
    \sqrt{v^2(f) \frac{d \log n}{n}}}
  & \le \sqrt{\frac{v^2(f)}{n^2}
    + \frac{\radphi^2 d \log n}{2 n^2}
    + \frac{v^2(f)}{2 n}} \\
  & \le \frac{\radphi \sqrt{d \log n}}{\sqrt{2} n}
  + \frac{v(f)}{\sqrt{n}},
\end{align*}
which implies the first statement of Lemma~\ref{lemma:peel-your-potatoes}.

The second statement follows via the same argument.

\bibliography{bib}
\bibliographystyle{abbrvnat}

\end{document}